\documentclass[a4paper]{amsart}
\usepackage{amsfonts}
\usepackage{amsmath}
\usepackage{mathrsfs}
\usepackage{graphicx}
\usepackage{amsmath,amsthm,amssymb,amscd}
\usepackage{enumerate}
\usepackage[T1]{fontenc}

\newtheorem{thm}{Theorem}[section]
\newtheorem{lem}[thm]{Lemma}

\newtheorem{defn}[thm] {Definition}
\newtheorem{ex}  [thm]{Example}
\newtheorem{cor}[thm]{Corollary}
\theoremstyle{remark}
\newtheorem{rem} [thm]{Remark}

\theoremstyle{definition}

\DeclareMathOperator{\supp}{supp}
\DeclareMathOperator{\diam}{diam}
\DeclareMathOperator{\dist}{dist}

\DeclareMathOperator{\Int}{Int}

\newcommand{\pw}{p\omega}
\newcommand{\cS}{\mathcal{S}}

\newcommand{\Leb}{\mathscr{L}}

\newcommand{\eps}{\varepsilon}
\newcommand{\R}{\mathbb{R}}
\newcommand{\N}{{\mathbb{N}}}
\newcommand{\Z}{\mathbb{Z}}
\newcommand{\set}[1]{\left\{#1\right\}}

\newcommand{\htop}{h_{\text{top}}}

\newcommand{\M}{\mathcal{M}}
\newcommand{\E}{\mathcal{E}}
\newcommand{\En}{\E_n}

\begin{document}

\title[On the irregular points for systems with the shadowing property]{On the irregular points for systems with the shadowing property}
%\title{Generically, regularity is observable and complexity is hidden.}
\author{Yiwei Dong}
\address[Y. Dong]{School of Mathematical Science,  Fudan University\\Shanghai 200433, People's Republic of China}
\email{yiweidong@fudan.edu.cn}
\author{Piotr Oprocha }
\address[P. Oprocha]{AGH University of Science and Technology\\
Faculty of Applied Mathematics\\
al. A. Mickiewicza 30, 30-059 Krak\'ow,
Poland\\ -- and ---\\
National Supercomputing Centre IT4Innovations, Division of the University of Ostrava,
Institute for Research and Applications of Fuzzy Modeling,
30. dubna 22, 70103 Ostrava,
Czech Republic} \email{oprocha@agh.edu.pl}
\author{Xueting Tian}
\address[X. Tian]{School of Mathematical Science,  Fudan University\\Shanghai 200433, People's Republic of China}
\email{xuetingtian@fudan.edu.cn}

\maketitle

\begin{abstract}
We prove that when $f$ is a continuous self-map acting on a compact metric space $(X,d)$ which satisfies the shadowing property, then the set of irregular points (i.e. points with divergent Birkhoff averages)
has full entropy.

Using this fact we prove that in the class of $C^0$-generic maps on manifolds, we can only observe (in the sense of Lebesgue measure)
points with convergent Birkhoff averages. In particular, the time average of atomic measures along orbit of such points converges to some SRB-like measure in the weak$^*$ topology. Moreover, such points carry zero entropy. In contrast, irregular points are non-observable
but carry infinite entropy.
\end{abstract}

\section{Introduction}
Consider a continuous self-map $f$ acting on a compact metric space $(X,d)$. For a continuous function $\varphi\colon X\to \R$, every point $x\in X$ and every positive integer $n$ we can calculate \textit{the Birkhoff average} $\frac{1}{n}\sum_{i=0}^{n-1}\varphi(f^i(x))$. According to the classical Birkhoff ergodic theorem, such average converges for all $x$ from a set carrying full measure for every $f$-invariant measure. We call such points \textit{$\varphi-$regular points}. The complementary set $I_\varphi(f)$ of points where Birkhoff average diverges is called \textit{$\varphi-$irregular set} and the union
$I(f):=\bigcup_{\varphi\in C(X,\mathbb{R})}I_{\varphi}(f)$ is called \textit{irregular set}.

From the viewpoint of the ergodic theory, the irregular sets $I_{\varphi}(f)$ and $I(f)$ are too small to be concerned. However, if we change measurement method, they may be quite large. To be concrete, let us recall some related results motivated in topological dynamics.

Fan \emph{et al} showed in \cite{FFW} that for mixing subshifts of finite type (i.e. subshifts with the shadowing property \cite{Walters2}) the set $I_{\varphi}(f)$ either is empty or has full entropy. Later on, Barreira and Schmeling \cite{Barreira-Schmeling2000} studied a broad class of systems including conformal repellers and horseshoes and showed that the irregular set carries full entropy and full Hausdorff dimension. Besides that, Olsen \cite{Olsen2002,Olsen2003,Olsen-Winter} introduced and developed a multi-fractal framework via deformations of empirical measures. Based on these results, Chen \emph{et al} \cite{CTS} considered irregular points for systems with the specification property. They used a method inspired by results of Takens and Verbitskiy from \cite{Takens-Verbitskiy} and showed that $I(f)$ has full topological entropy. Aside from dimensional theory, 
Li and Wu \cite{Li-Wu2014} proved that for systems with the specification property, $I_{\varphi}(f)$ either is empty or residual, which shows that $I_\varphi(f)$  and $I(f)$ are large from topological point of view (cf. results in \cite{DGS} on points with maximal oscillation). % and their consequence for the result in \cite{Li-Wu2014}). 
In \cite{FKKL} authors provide example with topological mixing, dense periodic points, countably many ergodic measures and empty irregular set (every point is generic for some ergodic measure).
%For topological pressure and for system with specification property, Thompson proved that $I_{\varphi}(f)$ either is empty or carries full topological pressure \cite{Thompson2010}.

All dynamical systems in the above mentioned results satisfy the specification property which was first introduced by Bowen to study Axiom A diffeomorphisms \cite{Bowen1971}. It is natural to ask what is the situation beyond the case of specification? More precisely, for which kinds of systems and to what extent can similar conclusions hold? Recently, Thompson gained some result in this direction. He considered systems with the almost specification property which is a natural weakening of  the specification property (see \cite{PS2007}) and proved in \cite{Thompson2012} that  $I_{\varphi}(f)$ either is empty or carries full entropy.

In this paper, we focus on another important property which motivated Bowen to define the specification property. Strictly speaking, we consider systems with the shadowing property, providing topological and measure-theoretic characterizations of $I(f)$ in their context.

Before proceeding, we should mention that neither of the properties of almost specification or shadowing implies the other. This can be seen from the following two examples.
%Let us recall some facts on (almost) specification and shadowing. Note that mixing dynamical system with shadowing property has the specification property \cite{DGS} and it is also not hard to see that specification property implies topological mixing \cite{DGS} for surjective maps.

%All these facts strongly rely on highly complicated global dynamics. However for systems with shadowing property, things may be  quite different. For concrete example, let us look at the following simple dynamical system with shadowing property which is neither topologically transitive (and thus not topologically mixing) nor entropy dense.

%Note that the above argument can be repeated for generic homeomorphisms with shadowing on $[0,1]$, since their dynamics is well understood \cite{Koscielniak2}. Therefore, Thompson's procedure \cite{Thompson2012} about $I_{\varphi}(f)$ cannot be repeated directly in this context. Nevertheless, we still may have some insight into the structure of $I(f)$ as shown later in the paper. Denote by $h_{top}(f)$ the classical topological entropy of $X$ defined in \cite{AKM} using open covers. There are also other ways to define entropy, which are not equivalent when calculated over subset but coincide on invariant sets (see Section \ref{pre} for definitions). First main result of the paper is the following:
\begin{ex}\label{shadowing-not-entropy-dense}
Consider a homeomorphism $f\colon [0,1]\to [0,1]$ such that $0,1$ are the only fixed points of $f$.
Then it is not hard to check that $f$ has the shadowing property (e.g. see \cite{CL}).
It is clear that the only invariant measures of $f$ are supported on fixed points $0,1$ and one of them is an attractor.
This immediately implies that $I(f)=\emptyset$ and one easily sees that $\frac{\delta_0+\delta_1}{2}\in \mathcal{M}_f([0,1])$ cannot be approximated by ergodic measures. Thus $([0,1],f)$ does not satisfy the almost specification property since for such systems, the ergodic measures are entropy dense \cite{PS2005}.
\end{ex}

\begin{ex}
It is well known that every $\beta$-shift has the almost specification property \cite{PS2005}, while the set of $\beta$ for which the $\beta$-shift has the specification property has zero Lebesgue measure \cite{Buzzi,Schmeling}. So let us consider some $\beta-$shift $\sigma\colon X_{\beta}\to X_{\beta}$ which satisfies the almost specification property but without the specification property. According to \cite[Theorem 3.8]{KKO}, when $f\colon X\to X$ is a continuous map with the shadowing property, the following two statement are equivalent:
\begin{itemize}
  \item $f$ is surjective and has the almost specification property;
  \item $f$ is surjective and has the specification property.
\end{itemize}
The $\beta$-shift is obviously surjective by definition \cite[p. 179]{Walters}. Hence $(X_{\beta},\sigma)$ may not have the shadowing property, because then it has the specification property which is a contradiction to its selection.
\end{ex}

\begin{rem}
It is proved in \cite[Theorem~3.8]{KKO} that totally transitive (i.e. $(X,f^n)$ transitive for every $n\geq 1$) maps with the shadowing property  have the specification property.
In particular, Thompson's results from \cite{Thompson2012} apply to this case.
\end{rem}

Some new obstacles appear compared to Thompson's result. For example there is no hope for transitivity in general as was the case in Example~\ref{shadowing-not-entropy-dense}.
 But even transitive but not mixing case is problematic, while it seems that relative almost specification property, as defined in \cite{KLO}, may still be sufficient to obtain Thompson's results.
	
\begin{ex}
Let $(X,f)$ be the dyadic adding machine (e.g. see \cite{Kurka} or \cite{Downar} for definition and basic properties) and let $(Y,g)$ be any mixing map with the shadowing property (e.g. full shift on $2$-symbols).
Note that every equicontinuous dynamical system on Cantor set has the shadowing property \cite{Kurka}, so $(X,f)$ has the shadowing property.
Then dynamical system $(X\times Y,f\times g)$ is transitive and has the shadowing property. But it does not have almost specification, since it is not mixing \cite{KKO}.
\end{ex}

To deal with cases beyond the specification property, in our approach we will focus on some ergodic measure $\mu$ which captures enough information of the entropy (by variational principle) and manipulate its points by the shadowing property.
%Thus it is possible to conduct our procedure using the transitivity of $\supp(\mu)$. Moreover, to create the irregular points we need, we also require another ergodic measure $\nu$ supported on $\mu$. This is another technique which is discussed at large in $\S$ \ref{section-lemmas} and concluded as corollary \ref{cor:measures_entropy}. Furthermore, we give a general illustration of the variational principle beyond the representative set $\supp(\mu)$ and put it as theorem \ref{thm:shadowingpte}. Finally, we use the above theorem to prove our main result which is stated as follows.
The main result of the paper is the following.
\begin{thm}\label{X-full}
If $(X,f)$ has the shadowing property then exactly one of the following conditions holds:
\begin{enumerate}
\item $\htop(f)>0$, $I(f)\neq \emptyset$ and has full entropy;
\item $\htop(f)=0$ and $I(f)=\emptyset$.
\end{enumerate}
\end{thm}

The above result will be the main tool in further studies on generic dynamical systems on some compact manifolds.
Before presenting our second result, we need some notions and notations in the first place.
%\section{Further results}

%This section is motivated by the $C^0$ generic results.
Throughout this paper, $M$  denotes a compact connected manifold admitting a decomposition in the following sense:
\begin{defn}
	\label{def:decomposition}
	A finite family $\cS$ of pairwise disjoint open subsets of $M$
	is called a \emph{decomposition} of $M$ if
	$M=\bigcup_{U\in \cS} \overline{U}$,
	and each $\overline{U}$ is homeomorphic to a closed ball in $\mathbb{R}^k$.
	If there exists a decomposition of $M$ then we say
	that $M$ \emph{admits a decomposition}.
\end{defn}
The class of manifolds admitting a decomposition includes all triangulable manifolds and manifolds with a handlebody.
Therefore, the considered class contains all compact manifolds
of dimension at most $3$ (see \cite{Moise,Bing}); of dimension at least $6$
(see \cite{Kirby-Siebenmann}) and all smooth manifolds
(see \cite{Cairns,Whitehead}).

Let $C(M)$ be the space of continuous maps on $M$ and $H(M)$ is the space of homeomorphisms on $M$. We consider these spaces with $C^0$-topology, that is we endow $C(M)$ with the metric
$$d_{C}(f,g)=\sup_{x\in M}d(f(x),g(x))$$
and $H(M)$ with the metric
$$d_H(f,g)=d_C(f,g)+d_C(f^{-1},g^{-1}).$$
It is well known that both spaces $(C(M),d_{C})$ and $ (H(M),d_{H})$ are complete.

A subset $\mathcal{R}$ of a metric space $X$ is \textit{residual} if it contains a countable intersection of dense
open sets.
%Recall that residual subsets of complete spaces are always dense by Baire theorem.
A property $\mathcal{P}$ is called \textit{generic} in the space $X$ if there is a residual subset $\mathcal{R}$ of $X$ such that any $f\in \mathcal{R}$ satisfies $\mathcal{P}$.
%One notes that for a complete metric space, a countable intersection of dense
%open sets is also dense. So a residual in a complete metric space is truly large from the topological viewpoint.
%
%Now recall that for our $M$, $C^0$ generic homeomorphism $f:M\to M$ has the shadowing property \cite{Pilyugin-Plamenevskaya}. Moreover, Koscielniak \emph{et al} \cite{KMO2014} recently consider a compact manifold $M$ with a decomposition and obtain that the C$^0$-generic continuous map $f\colon M\to M$ has the shadowing property. Note that triangulable manifolds admit a decomposition and all smooth manifolds are triangulable \cite{Cairns,Whitehead}. So for our $M$, $C^0$ generic continuous $f\colon M\to M$ also has the shadowing property.
Recent results show that in the class of manifolds with a decomposition shadowing property is generic (e.g. see \cite{Pilyugin-Plamenevskaya,KMO2014}).
%Combining the above facts and our main theorem, we are able to march to the $C^0$ case. Before doing that, some definitions should be made clear.

Let us denote by
$$\E_{n}(x):=\frac{1}{n}\sum_{j=0}^{n-1}\delta_{f^{j}(x)}$$
the $n$-th empirical measure of $x\in X$, where $\delta_{y}$ is the Dirac measure at $y\in X$. The points $x$ such that $\mathcal{E}_n(x)$ converges in weak$^*$ topology are called \textit{quasi-regular points} (cf. \cite{DGS} or \cite{Walters}) and the set of all quasi-regular points is denoted by $Q(f)$.
%\begin{defn}\label{def-of-SRB-like}

For  $x\in X$, let $\pw(x)$ denote the set of all limit points of the sequence $\mathcal{E}_{n}(x)$ in $\mathcal{M}(X)$. A probability measure $\mu\in \mathcal{M}_f(M)$ is  \emph{SRB}    (or  \emph{physical}) if  the set
$$B (\mu)=\{x\in M\colon  \pw (x)=\{\mu\}\}$$
has positive Lebesgue measure %. There is a generalized notion introduced in \cite{CE}, namely, the \emph{SRB-like} measure. A probability measure $\mu\in \mathcal{M}_f(M)$ is
and \textit{SRB-like} (alternatively,  \emph{observable} or  \emph{pseudo-physical}) if for any $\varepsilon>0$ the set
$$A_\varepsilon (\mu)=\{x\in M\colon  \rho(\pw (x),\mu)<\varepsilon\}$$
has positive Lebesgue measure.

Abdenur and Andersson showed in \cite{Abdenur-Andersson} that on manifolds with $\dim M\geq 2$, for $C^0$ generic maps $f\colon M\to M$ (the same for homeomorphisms), Lebesgue \textit{a.e.} $x\in M$ is quasi-regular, but $f$ admits no SRB measure.

However, for any topological dynamical system $f\colon M\to M$, the SRB-like measure always exists \cite{CE}. Moreover, let  $\mathcal{O}_{f}$ denote  the set of all SRB-like measures for $f$.
Define the basin $ \Delta$ of attraction of $\mathcal{O}_{f}$ by putting
 \begin{equation}
 \Delta=\{x: \pw(x)\subseteq \mathcal{O}_{f}\}.\label{def:delta}
 \end{equation}
It was proved in \cite[Theorem 1.5]{CE} that $\mathcal{O}_{f}$ is the smallest weak$^*$ compact subset of $\mathcal{M} (M)$
whose basin of attraction $\Delta$ satisfies $\Leb(\Delta)=1$, where $\Leb$ denotes the Lebesgue measure on $M$. Motivated by \cite{CE} and with Theorem~\ref{X-full} at hand
we prove the following.

\begin{thm}\label{main-thm}
For generic $f\in C(M)$ (the same for generic $f\in H(M)$) with $\dim M\geq 2$ the following conditions hold:
\begin{enumerate}

\item\label{main-thm:2} there is a set $\Lambda \subset M$ with $\Leb(\Lambda)=1$ such that
\begin{enumerate}[(i)]

  \item $\Lambda\subset Q(f)\cap \Delta$;
  \item $\htop(f,\Lambda)=0$;
\end{enumerate}
\item\label{main-thm:1} $\htop(f,I(f))=\infty$.
\end{enumerate}
%Here $\Leb$ is the Lebesgue measure on $M$ given by definition \ref{Lebesgue measure}.
\end{thm}

%This paper is organized as follows. In $\S$ \ref{pre}, we recall various basic notions and definitions known from the literature. In $\S$ \ref{section-lemmas} we prove several auxiliary lemmas which are used later in the paper. In $\S$ \ref{sec:X-full}, we combine all these facts together and prove Theorem~\ref{X-full}. Next we use Theorem~\ref{X-full} as a tool and prove Theorem~\ref{main-thm} in $\S$ \ref{proof-for-main}. Finally, we propose some questions for further research in $\S$ \ref{questions}.

\section{Preliminaries}\label{pre}
\subsection{Notions and notations}\label{notions-and-notations}
Throughout this paper, we denote by $\N,\Z,\N_0$ the set of natural numbers, integers, and nonnegative integers respectively. A \textit{dynamical system} is a pair $(X,f)$ consisting of a compact metric space $(X,d)$ and a continuous map $f\colon X\to X$.
If for every pair of non-empty open sets $U,V$ there is an integer $n$ such that $f^n(U)\cap V\neq \emptyset$ then $(X,f)$ is \textit{transitive}.
When for every pair of non-empty open sets $U,V$ there exists an integer $N$ such that $f^n(U)\cap V\neq \emptyset$ for every $n>N$, then $(X,f)$ is \textit{topologically mixing}.

We denote the diameter of any subset $A\subset X$ by
$$\diam(A)=\sup_{x,y\in A}\{d(x,y)\}.$$
An open ball, centered at $x\in X$ and with radius $r>0$ is denoted by $B(x,r)$. For any $n\in \mathbb{N}$, the $d_n$-distance between $x,y\in X$ is defined as
$$d_n(x,y):=\max_{0\leq i\leq n-1}\{d(f^{i}(x),f^{i}(y))\}.$$
The $(n,\eps)$-Bowen ball centered at $x$ is
$$B_n(x,\eps):=\{y\in X:d_n(x,y)<\eps\}.$$

\subsection{Invariant measures}
We denote by $\mathcal{M}(X)$ the set of Borel probability measures on $X$.
Let ${\{\varphi_{j}\}}_{j\in\mathbb{N}}$ be a dense subset of $C(X,\mathbb{R})$, then
$$\rho(\xi,\tau)=\sum_{j=1}^{\infty}\frac{|\int\varphi_{j}d\xi-\int\varphi_{j}d\tau|}{2^{j}\|\varphi_{j}\|}$$
defines a metric on $\mathcal{M}(X)$ for the $weak^{*}$ topology  \cite{Walters}, where
$$\|\varphi_i\|=\max\{|\varphi_i(x)|:x\in X\}.$$
Note that
\begin{equation}\label{diameter-of-Borel-pro-meas}
\rho(\xi,\tau)\leq2~~\textrm{for any}~~\xi,\tau\in \mathcal{M}(X).
\end{equation}
Note also that the natural embedding $j:x\mapsto \delta_x$ is continuous and injective. %. Meanwhile, if $\delta_y=\delta_z$, then for every $\varphi\in C(X,\R)$, %$\varphi(y)=\varphi(z)$. In particular, by letting $\varphi(x)=d(x,z)$, we have $d(y,z)=d(z,z)=0$, and hence $y=z$, which implies that $j$ is injective. 
Moreover, $(X,f)$ is compact and $\M(X)$ is Hausdorff, so $X$ is homeomorphic to its image $j(X)$. Thus without loss of generality,
we can choose the following equivalent metric $d$ on $X$ putting $d(x,y)=\rho(\delta_x,\delta_y)$.
This combined with \eqref{diameter-of-Borel-pro-meas} gives
\begin{lem}\label{lem:prohorov}
	Let $(X,f)$ be a dynamical system and let $x\in X$.
	\begin{enumerate}
		\item\label{lem:prohorov:1} Let $0\leq k< m <n$ and $x\in X$. Then
		$$\rho(\E_m(x),\E_n(f^k(x)))\leq \frac{2}{n}(n-m+k),$$
		\item\label{lem:prohorov:2} Given $\eps > 0$ and $p\in \N$, for every $y  \in B_p(x,\eps)$ we have $\rho(\E_p(y), \E_p(x)) <\eps$.
		\item\label{lem:prohorov:3} Given $\eps > 0$ and $p,q \in \N$ satisfying $p \leq q \leq (1+\eps/2)p$, for every $y  \in B_p(x,\eps)$ we have $\rho(\E_q(y), \E_p(x)) <2\eps$.
	\end{enumerate}
\end{lem}
We say that $\mu\in \M(X)$ is $f$-invariant if for any measurable set $E\subset X$, we have
$$\mu(E)=\mu(f^{-1}(E))$$
that is, $\mu$ is a fixed point of $T_f$, where $T_f:\M(X)\to \M(X)$ is defined as $T_f(\mu)=\mu\circ f^{-1}$.
The set of $f$-invariant measures in $\mathcal{M}(X)$ is denoted by $\mathcal{M}_f(X)$.
A dynamical system $(X,f)$ is \textit{uniquely ergodic} when $\mathcal{M}_f(X)$ is a singleton.

\subsection{Shadowing property}

A sequence $\{x_n\}_{n\in\mathbb{N}}\subset X$ is called a \textit{$\delta$-pseudo-orbit} of $f$ if
$$d(f(x_n),x_{n+1})<\delta~~\textrm{for any}~~n\in\mathbb{N}.$$
A $\delta$-pseudo-orbit $\{x_n\}_{n\in\mathbb{N}}$ is \textit{$\eps$-shadowed} by an orbit of $y\in X$ if
$$d(f^n(y),x_n)<\eps~~\textrm{for any}~~n\in\mathbb{N}.$$
We say that $f\colon X\to X$ has the \textit{shadowing property} if for any $\eps>0$, there exists  $\delta>0$ such that any $\delta$-pseudo-orbit is $\eps$-shadowed by some orbit.

\subsection{Topological entropy}\label{topological-entropy}
%The topological entropy for dynamical systems $(X,f)$ was first introduced by Adler-Konheim-McAndrew \cite{AKM} using open covers which we denote by $\htop(f)$ here. Later on, Bowen and Dinaburg independently reformulated Adler-Konheim-McAndrew definition, using the notion of separated and spanning sets \cite{Din,BowenEntropy}. Here we follow notation from \cite{BowenEntropy}.
The topological entropy of dynamical systems $(X,f)$ was first defined by Adler, Konheim and McAndrew in \cite{AKM} using open covers. Later on, Bowen introduced separated and spanning sets to reformulate this notion \cite{BowenEntropy}.

Let $Z\subset M$. A set $S$ is $(n,\eps)$-separated for $Z$ if $S\subset Z$ and $d_n(x,y)>\eps$ for any distinct points $x,y\in S$.
A set $S\subset Z$ if $(n,\eps)$-spanning for $Z$ if for any $x\in Z$, there exists $y\in S$ such that $d_n(x,y)\leq \eps$.

Define
\begin{eqnarray*}
	s_n(Z,\eps) &=& \sup ~\{|S|:S ~\textrm{is}~ (n,\eps)-\textrm{separated for}~ Z\}, \\
	r_n(Z,\eps) &=& \inf ~\{|S|:S ~\textrm{is}~ (n,\eps)-\textrm{spanning for} ~Z\},
\end{eqnarray*}
where $|S|$ denotes the cardinality of $S$. It is not hard to check that (e.g see \cite{Walters}):
\begin{equation}\label{inequality-of-rn-and-sn}
r_n(Z,\eps)\leq s_n(Z,\eps)\leq r_n(Z,\eps/2).
\end{equation}
Following Bowen, we define \emph{topological entropy} for a compact set $K\subset X$ by
$$h_d(f,K)=\lim_{\eps\to0^+}\limsup_{n\to\infty}\frac1n\log s_n(K,\eps)=\lim_{\eps\to0^+}\limsup_{n\to\infty}\frac1n\log r_n(K,\eps).$$

Bowen also introduced topological entropy for non-compact set using notation similar to Hausdorff dimension \cite{Bowen}.
Let $E\subset X$, and $\mathcal {G}_{n}(E,\eps)$ be the collection of all finite or countable covers of $E$ by sets of the form $B_{u}(x,\eps)$ with $u\geq n$. Put
$$C(E;t,n,\eps,f):=\inf_{\mathcal {C}\in \mathcal {G}_{n}(E,\eps)}\sum_{B_{u}(x,\eps)\in \mathcal {C}}e^{-tu}$$
and
$$C(E;t,\eps,f):=\lim_{n\rightarrow\infty}C(E;t,n,\eps,f).$$
Then we define
$$\htop(E;\eps,f):=\inf\{t:C(E;t,\eps,f)=0\}=\sup\{t:C(E;t,\eps,f)=\infty\}$$
%\begin{defn}
The \textit{Bowen topological entropy} of $E$ is
\begin{equation}\label{definition-of-topological-entropy}
  \htop(f,E):=\lim_{\eps\rightarrow 0^+} \htop(E;\eps,f).
\end{equation}

It is known that value of $h_d$ is independent of the choice of metric $d$ (it depends only on topology on $X$) and
for every compact $f$-invariant subset $E\subset X$, we have $\htop(f,E)=h_d(f,E)$, see \cite{Pesin}.

\subsection{Symbolic dynamics}
Fix any positive integer $k\geq 2$ and consider the following set
$\Sigma_k^+=\{0,1,\dotsc,k-1\}^{\N_0}$ with the product topology induced by the discrete topology on $\{0,1,\dotsc,k-1\}$.
The space $\Sigma_k^+$ is always endowed with the \textit{shift map} $\sigma$ defined by $\sigma(x)_i=x_{i+1}$ for every integer $i\geq 0$.
It is not hard to verify that $\sigma$ is continuous and $\Sigma_k^+$ is a compact metrizable space. We endow it with the (compatible) metric
defined by
$d(x,y)=2^{-k}$ when $x\neq y$ and $k=\min\{i : x_i\neq y_i\}$. Dynamical system $(\Sigma_k^+,\sigma)$ or simply $\Sigma_k^+$ for short,
is called \textit{full shift}. By \textit{subshift} or \textit{shift} we mean any compact and $\sigma$-invariant subset of $\Sigma_k^+$.

For $i\leq j$ and any $x\in \Sigma_k^+$ we write $x_{[i,j]}=x_i\ldots x_j$ and similarly $x_{[i,j)}=x_{[i,j-1]}$. For any $n>0$ and any subshift $X$ we denote
$$
B_n(X)=\{x_{[i,i+n)} : x\in X, i\geq 0\}.
$$
It is well known that in the case of subshifts, formula for topological entropy reduces to
$$
\htop(\sigma,X)=\lim_{n\to\infty}\frac{1}{n}\log |B_n|.
$$
The reader not familiar with symbolic dynamics is referred to \cite{LM} for more details.

\section{Some auxiliary lemmas}\label{section-lemmas}
We start this section with the following simple observation, relating quasi-regular points with irregular set.
\begin{lem}\label{lem:quasi}
Let $(X,f)$ be a dynamical system. Then
$Q(f)=X\setminus I(f)$.
\end{lem}
\begin{proof}
Let $x\in Q(f)$ be a quasi-regular point for a measure $\mu\in \M(X)$.
	Observe that $\lim_{n\to \infty} \E_n(x)=\mu$ in a weak$^*$ topology if
	$$
	\lim_{n\to \infty} \frac{1}{n}\sum_{i=0}^{n-1} \varphi(f^i(x))=\lim_{n\to \infty}\int \varphi d\E_n(x)=\int \varphi d\mu
	$$
	for every $\varphi \in C(X, \R)$, which means $x\not\in I(f)$.
	On the other hand, if $x\not\in I(f)$, then for every $\varphi \in C(X,\R)$ the number $L_f(\varphi)=\lim_{n\to\infty} \frac{1}{n}\sum_{i=0}^{n-1} \varphi(f^i(x))$ is well defined.
	Clearly $L\colon C(X,\R)\to \R$ is positive operator, $L(1)=1$ and $L(\varphi-\psi)\leq  \sup_{x\in X}|\varphi(x)-\psi(x)|$ which shows that $L$ is continuous.
	Then by Riesz representation theorem there is $\mu \in M(X)$ such that $L(\varphi)=\int \varphi d\mu$ and so $x\in Q(f)$.
\end{proof}

In our constructions we will need the following combinatorial lemma which allows us to construct large separated sets.
\begin{lem}\label{useful-proposition}
	Let $(X,d)$ be a compact metric space and let measure $\mu\in\mathcal{M}_f(X)$ be ergodic with $h_\mu(f)>0$. Then for any $\eta>0$, there exists $\eps>0$ such that for each neighborhood $F$ of $\mu$ in $\M(M)$, there exists $n_F\in\mathbb{N}$ such that for any $n\geq n_F$, there exists $\Gamma_n\subset X_{n,F}\cap \supp(\mu)$ which is $(n,\eps)$-separated and satisfies $|\Gamma_n|\geq e^{n(h_{\mu}(f)-\eta)}$, where $X_{n,F}:=\{x\in X:\En(x)\in F\}$.
\end{lem}
\begin{proof}
	Fix any ergodic measure $\mu\in \M_f(X)$. For any choice of $n\geq 0$ and $\eps,\delta>0$,
let $N_f(n,\eps,\delta)$ denote the minimal number of $(n,\eps)$-Bowen balls which cover the set of $\mu$-measure at least $1-\delta$. Using the above notation, Katok proved in \cite[Theorem~1.1]{Katok} the following generalization of Bowen's formula for topological entropy:
	
\begin{equation}\label{Katok-entropy-formula}
	h_{\mu}(f)=\lim_{\eps\rightarrow 0^+}\liminf_{n\rightarrow \infty}\frac{\ln N_f(n,\eps,\delta)}{n},
	\end{equation}

Fix any $\eta>0$ and let $\eps>0$ be such that
\begin{equation}\label{Katok-entropy-formula:2}
	h_{\mu}(f)-\eta< \liminf_{n\rightarrow \infty}\frac{\ln N_f(n,\eps,\delta)}{n}.
	\end{equation}
Fix any open neighborhood $F$ of $\mu$ in $\M(X)$. Since $\mu$ is ergodic, the set $G_\mu$ of points generic for $\mu$ has full $\mu$-measure and by definition for every generic point $x$ we have $\lim_{n\to\infty}\En(x)=\mu$.
Note that $G_\mu \subset \bigcup_n D_n$ where
$$D_n=\set{x\in G_\mu : \E_k(x)\in F~\text{ for all}~k\geq n}.$$
Note that $D_{n}\subset D_{n+1}$ and $D_n\subset X_{n,F}$, hence
	there exists $N_1\in\N$ such that
	$\mu(X_{n,F}\cap \supp(\mu))\geq1-\delta$ for every $n>N_1$.
	
By \eqref{Katok-entropy-formula:2} there exists $n_F\geq N_1$ such that for any $n\geq n_F$ we have:
	\begin{equation}\label{Katok-spanning-set}
	N_f(n,\eps,\delta)\geq e^{n(h_{\mu}(f)-\eta)}.
	\end{equation}
	
	Denote $Z=X_{n,F}\cap \supp(\mu)$ and observe that \eqref{Katok-spanning-set} implies that
	$$r_n(Z,\eps)\geq e^{n(h_{\mu}(f)-\eta)}.$$
	
	By the inequality \eqref{inequality-of-rn-and-sn}, there exists $(n,\eps)$-separated set $\Gamma_n\subset Z$ which satisfies
	$$|\Gamma_n|\geq e^{n(h_{\mu}(f)-\eta)}.$$
	The proof is completed.
\end{proof}

\begin{lem}\label{two-measures}
	Let $(X,f)$ be a dynamical system.
	Fix any $\alpha>0$ and assume that there are $m,k\in \N$ with $\log (m)/k>\alpha$ and a closed set $\Lambda\subset X$
	invariant under $f^k$ such that there is a factor map $\pi\colon (\Lambda, f^k)\to (\Sigma_{m+1}^+,\sigma)$.
	Then there exists an ergodic measure $\mu$ such that
	\begin{equation}\label{estimation-of-entropy-of-mu}
	h_{\mu}(f)>\alpha
	\end{equation}
	and there is an ergodic measure $\nu$ such that $\mu\neq \nu$ and $\supp(\nu)\subset \supp(\mu)$.
\end{lem}
\begin{proof}	
	Recall that dynamical system $(X,f)$ is proximal, if
$$\liminf_{n\to\infty} d(f^n(x),f^n(y))=0$$ for every $x,y\in X$. Equivalently, it means that the only minimal subset of $X$ is the unique fixed point in $X$ (e.g. see \cite{Akin}).
	
	Take any $\gamma>0$ such that $(1-\gamma)\log(m)/k>\alpha$.
	%We may also assume that $(1-\gamma)\log(m)>\log(m-1)$.
	Fix any increasing sequences of positive integers $k_n,s_n$ such that $s_n$ divides $s_{n+1}$, $k_{n+1}<s_n$ and
	$$
	\sum_{i=1}^\infty \frac{k_i}{s_i}< \gamma.
	$$
Now, let $A$ be a set consisting of points $x\in \Sigma_{m+1}^+$  such that
	$x_i=0$ when  $i (\text{mod }s_n)\geq (s_n-k_n)$ and $x_i \in \set{1,\ldots,m}$
	otherwise. Denote $X=\overline{\bigcup_{i=0}^\infty\sigma^i(A)}$. Since every element of $A$ has syndetic occurrence of block $0^{k_n}$ (next such block occurs after at most $2s_n$ positions), also every element of $X$ has this property. It immediately implies that $0^\infty \in \omega(x)$ for every $x\in X$, hence $\set{0^\infty}$ is the unique minimal subset of $X$. Therefore dynamical system $(X,\sigma)$ is proximal.
	Let us estimate its entropy.
	\begin{eqnarray*}
		\frac{1}{s_n}\log |B_{s_n}(X)|&\geq & \frac{1}{s_n}\log |\set{ w : w=x_{[0,s_n)]}, x\in A}|\\
		&\geq& \frac{1}{s_n} \log m^{(s_n - \sum_{i=1}^n \frac{k_i s_n}{s_i})}\\
		&\geq & (1-\sum_{i=1}^\infty \frac{k_i}{s_i})\log (m)\geq (1-\gamma)\log(m).
	\end{eqnarray*}
	This shows that $\htop(X)\geq (1-\gamma)\log(m)$.
	
	Denote $\Gamma=\pi^{-1}(X)$ and let $Z=\bigcup_{i=0}^\infty f^i(\Gamma)=  \bigcup_{i=0}^{k-1}f^i(\Gamma)$. Then
	$$
	\htop(f|_Z)\geq \frac{1}{k}\htop(f^k|_{\Gamma})\geq (1-\gamma)\log(m)/k>\alpha.
	$$
	Denote by $D$ the smallest invariant subset for $f$ containing the set $\pi^{-1}(0^\infty)$.
	Then clearly $D\subset Z$ and every minimal system $E$ in $Z$ must be a subsystem of $D$. Simply $E\cap Z\neq \emptyset$ and  $E\cap Z$ is a union of minimal sets for $f^k$. But then the only possibility is that $\pi(E\cap Z)=0^\infty$ showing that $E\cap Z\subset D$, and since $D$ is invariant we have $E\subset D$.
	
	Let $(Y,g)$ be a system obtained by collapsing $D$ to a single point, where $g$ is an appropriate quotient map obtained from $f$. Then we have a natural factor map $\eta \colon (Z,f)\to (Y,g)$. Let $\tilde{\Gamma}$ be obtained from $\Gamma$ by collapsing $D\cap \Gamma$. Then $\tilde{\Gamma}\subset Y$ and we also obtain an induced map $\tilde{\eta}\colon (Y,g)\to (X,\sigma)$ such that $\pi = \tilde{\eta}\circ \eta$.
	Among other things, it implies that $\htop(g)>\alpha$. Let $\mu$ be an ergodic measure with $h_\mu(g)>\alpha$ obtained by the variational principle. Note that $(Y,g)$ is proximal, since there are no minimal subsets in it other that the fixed point $p$ obtained by the collapse of $D$ in $Z$. But then, by ergodic theorem $\mu(\set{p})=0$. Observe that $\eta$ is one-to-one everywhere outside $D$ and hence we may view $\mu$
	as a measure on $Z$ with $\mu(D)=0$. Clearly this backward projection by $\eta$ does not change entropy of $\mu$.
	To complete the proof, it is enough to take as $\nu$ any ergodic measure supported on some minimal set contained in the invariant set $D\cap \supp(\mu)\neq \emptyset$.
\end{proof}
\begin{lem}\label{horseshoe}
	Suppose that a dynamical system $(X,f)$ acting on a compact metric space $X$ has the shadowing property and suppose that  there exists $\alpha>0$ and an ergodic measure  $\mu\in\mathcal{M}_f(X)$ with $h_\mu(f)>\alpha$.
	 Then there are $m,k\in \N$, $\log (m)/k>\alpha$ and a closed set $\Lambda\subset X$
	 invariant under $f^k$ such that there is a factor map $\pi\colon (\Lambda, f^k)\to (\Sigma_{m+1}^+,\sigma)$
\end{lem}
\begin{proof}
Fix any $\eta>0$ such that $h_\mu(f)-4\eta>\alpha$. By Lemma~\ref{useful-proposition} there exists $\eps>0$ and $N\in\mathbb{N}$ such that for any $n\geq N$, there exists $\Gamma_n\subset \supp(\mu)$ which is $(n,\eps)$-separated and satisfies $|\Gamma_n|\geq e^{n(h_{\mu}(f)-\eta)}$.

Let $\delta>0$ be such that every $\delta$-pseudo orbit is $\eps/3$ traced and cover $\supp(\mu)$ by opens sets $U_1,\ldots,U_s$
such that $\diam(U_i)<\delta$ for every $i$. Since $\mu$ is ergodic, $f|_{supp(\mu)}$ is transitive thus for every $1\leq i,j\leq s$, there exists a point
$z_{ij}\in U_i$ such that $f^{k_{ij}}(z_i)\in U_j$ for some $k_{ij}\in \mathbb{N }$. Denote
$$K=\max_{1\leq i,j\leq s}\{k_{ij}\}.$$

Now we choose $n\in \mathbb{N}$ large enough such that
\begin{eqnarray}
% \nonumber to remove numbering (before each equation)
  \label{s-n}2\log(s)/n &<& \eta, \\
  \frac{n}{n+K}h_\mu(f) &<& h_{\mu}(f)-\eta.
\end{eqnarray}
For every $1\leq i,j\leq s$ put $\Gamma_n^{ij}=\set{z\in \Gamma_n : z\in U_i, f^n(z)\in U_j}$. Fix some $1\leq i,j\leq s$
such that $|\Gamma_n^{ij}|\geq |\Gamma_n|/s^2$ and observe that by \eqref{s-n}, we have
$$
|\Gamma_n^{ij}|\geq e^{n(h_{\mu}(f)-2\eta)}.
$$

Enumerate elements of $\Gamma_n^{ij}$,
say $\Gamma_n^{ij}=\{p_0,\ldots,p_{r-1}\}$ where $r=|\Gamma_n^{ij}|$. Using these points define $r$ finite sequences,  putting for $l=0,1,\ldots,r-1$:
$$	
\eta(l)=(p_l,f(p_l),\dotsc,f^{n-1}(p_l)),z_{ji}, f(z_{ji}),\ldots, f^{k_{ji}-1}(z_{ji}).
$$
If we denote by $k$ length $k=|\eta(l)|=n+k_{ij}\leq n+K$ then all sequences $\eta(l)$
are $\delta$-pseudo orbits, and since $f^{k_{ji}}(z_{ji})\in U_i \supset \Gamma_n^{ij}$
we can freely concatenate sequences $\eta(l)$ obtaining another, longer $\delta$-pseudo orbits.

Let $\Sigma_r^+$ be the set of element $a=(a_0a_1\dotsc a_n\dotsc)$ such that $a_i\in \{p_0,\ldots,p_{r-1}\}$, $i\in \mathbb{Z}^+$. For every $\xi\in\Sigma_{r}^+$,
denote
$$
Y_\xi=\{z : d(f^{ik}(z),\xi_i)\leq \eps/3~\textrm{for}~i\in\mathbb{Z}^+\}.
$$
By the choice of $\delta$ and the shadowing property of $(X,f)$ we see that each set $Y_\xi$ is nonempty.
It is also clear that each such set is a closed subset of $X$.
Note that if $\xi\neq \psi$ then there is $t$ such that $\xi_t \neq \psi_t$.
But then there is $r<n$ such that $d(f^r(p_{\xi_t}),f^r(p_{\psi_t}))\geq \eps$. This immediately implies that $Y_\xi\cap Y_\psi=\emptyset$. Denote
$$
\Lambda=\bigcup_{\xi\in\Sigma_r^+} Y_\xi
$$
Note that $f^{k}(Y_\xi)\subset Y_{\sigma(\xi)}$, then $\Lambda$ is an invariant set for $f^k$.
It is also not hard to see that if $y\in Y_\xi$, $z\in Y_\psi$ and $d(f^l(x),f^l(y))<\eps/3$
for $l=0,\ldots, ks-1$ then $\xi_i=\psi_i$ for $i=0,\ldots,s-1$. So we can write $\Lambda$ as the disjoint union of $Y_{\xi}$:
$$\Lambda=\bigsqcup_{\xi\in\Sigma_r^+}Y_{\xi}.$$
Therefore, if we define $\pi \colon \Lambda \to \Sigma_{r}^+$ as
$$\pi(x):=\xi\quad \textrm{ for all } x\in Y_{\xi},$$
then $\pi$
is a continuous surjection. This shows that $\Lambda$ is closed and clearly also $\sigma \circ \pi =\pi \circ f^k$.
Finally put $m=r-1$ and observe that
$$
\frac{\log(m)}{k}=\frac{\log(r-1)}{k}\geq\frac{\log(r/2)}{k}=\frac{1}{k} n(h_\mu(f)-3\eta)\geq h_\mu(f)-4\eta>\alpha.
$$
The proof is completed.
\end{proof}

Then it is enough to combine the Variational Principle with Lemma~\ref{horseshoe} and Lemma~\ref{two-measures} to get the following corollary.
\begin{cor}\label{cor:measures_entropy}
If $(X,f)$ satisfies the shadowing property and $\htop(f)>0$, then there exists a sequence of 	ergodic measures $\set{\mu_n}_{n=1}^\infty \subset \mathcal{M}_f(X)$ such that
$\htop(f)=\lim_{n\to\infty}h_{\mu_n}(f)$ and for each
$n$ there is an ergodic measure $\nu_n$ such that $\mu_n\neq \nu_n$ and $\supp(\nu_n)\subset \supp(\mu_n)$,
i.e. $f$ restricted to $\supp(\mu_n)$ is not uniquely ergodic.
\end{cor}

\section{Proof of Theorem~\ref{X-full}}\label{sec:X-full}
%
%Let $\Lambda\subset X$ be a closed invariant set. Then $f|_{\Lambda}$ is a subsystem.
First we prove the following theorem which covers a part of Theorem~\ref{X-full}.
\begin{thm} \label{thm:shadowingpte}
If $f$ satisfies the shadowing property and $\htop(f)>0$, then
\begin{eqnarray*}
\htop(f)&= &\htop(f,I(f))\\
&=& \sup \{\htop(f|_{\Lambda})~:~\Lambda\subset X ~\textrm{is a closed invariant set with}\\
& &\qquad\qquad f|_{\Lambda}~\textrm{transitive and not uniquely ergodic}\}.
\end{eqnarray*}
\end{thm}
\begin{proof}
First we consider the case of $\htop(f)<\infty$ and divide our proof into the following four steps.
\begin{enumerate}[1.]
  \item For any $\eta>0$ find $\Lambda$ with $f|_\Lambda$ transitive, non-uniquely ergodic, $\htop(f)-\eta<\htop(f|_\Lambda)$ and pick a sufficient number of pseudo-orbits;
  \item Construct an uncountable compact set $G$;
\item Prove that $G\subset I(f)$;
%  \item Prove that $h_d(f,G)>\htop(f|_{\Lambda})-5\eta$.
  \item Prove that $\htop(f,G)>\htop(f|_{\Lambda})-6\eta$.
\end{enumerate}
%Since we can apply this procedure for any $\eta>0$, by Lemma~\ref{lem:entropy}\eqref{ine-of-top-entropy}  the result follows.

\textbf{Step 1}. For any $\eta>0$, by Corollary~\ref{cor:measures_entropy} we select an ergodic measure $\mu\in\mathcal{M}_f(\Lambda)$ such that
$$h_{\mu}(f)>\htop(f)-\eta$$
and another ergodic measure $\nu$ with $\supp(\nu)\subset \supp(\mu)$. Clearly $f$ restricted to invariant set $\Lambda=\supp{\mu}$ is transitive (e.g. see \cite{Walters}).
Denote $d(\mu,\nu)=\alpha$ and let $t\in \mathbb{N}$ be such that
	\begin{equation}\label{estimation-of-t}
	t(h_{\mu}(f)-3\eta)\geq(t+1)(h_{\mu}(f)-4\eta).
	\end{equation}
	
By Lemma~\ref{useful-proposition}, for $\mu$, there exist $\eps>0$ such that if we denote  $F_{1}=\mathcal{B}(\mu,\frac{\alpha}{6(t+1)})\subset \mathcal{M}(X)$ then there exists $m_{F_1}\in\mathbb{N}$ such  that for any $n\geq m_{F_{1}}$, there exists $\Gamma_n^{\mu}\subset X_{n,F_{1}}\cap \supp(\mu)$ which is $(n,\eps)$-separated and satisfies $|\Gamma_n^{\mu}|\geq e^{n(h_{\mu}(f)-\eta)}$.

Denote $F_{2}=\mathcal{B}(\nu,\frac{\alpha}{6(t+1)})$ and let $\Gamma_n^{\nu}=\set{z}$ for some generic point $z$ for $\nu$. Then there exists $m_{F_2}\in\mathbb{N}$ such that $\Gamma_n^{\nu}\subset X_{n,F_{2}}\cap \supp(\nu)$ for any $n\geq m_{F_{2}}$.
%Since $\Gamma_n^\mu$ is a singletowhich is $(n,\eps_2)$-separated and satiesfies $|\Gamma_n^{\nu}|\geq e^{n(h_{\nu}(f)-\eta)}$. Let $\eps=\min\{\eps_1,\eps_2\}$, then $\Gamma^{\mu}_n$ and $\Gamma^{\nu}_n$ are both $(n,\eps)$-separated.
	
By the shadowing property, for
\begin{equation}\label{selection-of-zeta}
	\zeta=\min\set{\frac{\alpha}{12(t+1)},\frac{\eps}{3}},
\end{equation}
there exists $\gamma>0$ such that any $\gamma$-pseudo-orbit can be $\zeta$-shadowed by some point.
	
Since $\Lambda$ is closed and thus compact, there exists a finite open cover $\{U_i\}_{i=1}^s$ of $\Lambda$ with $\diam(U_i)<\gamma$ for $i=1,\cdots,s$.
Since $f|_{\Lambda}$ is transitive, for any $1\leq i,j\leq s$, there exist $n_{ij}\in\mathbb{N}$ and %an orbit segment $\Delta_{ij}=\{f^k(y_{ij})\}_{k=0}^{n_{ij}}$ such that
$y_{ij}\in U_i\cap \Lambda$ such that $f^{n_{ij}}(y_{ij})\in U_j$. Denote
\begin{eqnarray*}
T&:=&\max_{1\leq i,j\leq s}\{n_{ij}\},\\
L&:=&\max\{m_{F_1},m_{F_2},\frac{48T(t+1)}{\alpha},\frac{T(h_{\mu}(f)-3\eta)}{\eta},\frac{2\ln (s)}{\eta},5T %, \frac{T}{\zeta}
\}.
\label{chosen-of-L}
\end{eqnarray*}
%In particular $L+l_{ij}<(1+\zeta)L$ for every $1\leq i,j\leq s$.
By the definition, there exist $(L,\eps)$-separated sets $\Gamma_L^{\mu}\subset X_{L,F_{1}}\cap \supp(\mu)$ with
	\begin{equation}\label{estimation-of-entropy}
	|\Gamma_L^{\mu}|\geq e^{L(h_{\mu}(f)-\eta)}.
	\end{equation}
and then by pigeonhole principle we obtain  $1\leq i_1, j_1\leq s$ and  $\widetilde{\Gamma}_L^{\mu}\subset \Gamma_L^{\mu}$  such that $\widetilde{\Gamma}_L^{\mu}\subset U_{i_1}$, $f^L(\widetilde{\Gamma}_L^{\mu})\subset U_{j_1}$
and
	\begin{equation}\label{cardinality-of-gamma-tilde-mu}
	|\widetilde{\Gamma}_L^{\mu}|\geq |\Gamma_L^{\mu}|/s^2\geq \frac{e^{L(h_{\mu}(f)-\eta)}}{s^2}\geq e^{L(h_{\mu}(f)-2\eta)}.
	\end{equation}
Take (the unique) $z\in \Gamma_L^{\nu}$ and let $1\leq i_2,j_2\leq s$ denote indexes such that
$z\in U_{i_2}$ and $f^L(z)\in U_{j_2}$.	
	
%	Similarly, there exists $\widetilde{\Gamma}_L^{\nu}\subset \Gamma_L^{\nu}$ and $1\leq i_2, j_2\leq N$ such that
%	$$\widetilde{\Gamma}_L^{\nu}\subset U_{i_2}~\textrm{and}~f^L(\widetilde{\Gamma}_L^{\nu})\subset U_{j_2}$$
%	with
%	\begin{equation}\label{cardinality-of-gamma-tilde-nu}
%	|\widetilde{\Gamma}_L^{\nu}|\geq |\Gamma_L^{\nu}|\geq \frac{e^{L(h_{\nu}(f)-\eta)}}{N^2}.
%	\end{equation}
		
\textbf{Step 2}.	Now, for $n\in \mathbb{N}$, let us inductively define
$$S_0=0,~N_1=1,~S_n=\sum_{i=1}^nN_i~\textrm{ and} ~N_n=\lambda S_{n-1} ~\textrm{for}~ n\geq2,$$
 where $\lambda >\frac{48(t+1)}{\alpha}$ is a fixed integer. Observe that we have an explicit formula $N_n=(\lambda+1)^{n-1}$.
For each $m\in \N$ define measure $\omega_m$ and integer $l_m$ by the following procedure.	
If $m\in [S_{2k}(t+1)+1,S_{2k+1}(t+1)]$ for some $k\geq 0,$ then we put
$$\omega_m=\mu~\textrm{and}~l_m=n_{j_1i_1}.$$
Otherwise, $m\in [S_{2k+1}(t+1)+1,S_{2k+2}(t+1)]$ for some $k\geq 0$ and then there are uniquely determined integers $r\in [0,N_{2k+1})$
and $\tilde{r}\in [1,t+1]$ such that $m=(S_{2k+1}+r)(t+1)+\tilde{r}$. In this case, we put
$$
\omega_m=\begin{cases}
\mu,& \textrm{if }\tilde{r}\leq t\\
\nu,& \textrm{if }\tilde{r}=t+1
\end{cases},\quad\quad
l_m=\begin{cases}
n_{j_1i_1},& \textrm{if }\tilde{r}<t\\
n_{j_1i_2},& \textrm{if }\tilde{r}=t\\
n_{j_2i_1},& \textrm{if }\tilde{r}=t+1\\
\end{cases}
$$
	
%	\begin{itemize}
%		\item For $m=(S_{2k}+r)(t+1)+\widetilde{r}$, $k\geq0$, $0\leq r<N_{2k+1}$, $1\leq \widetilde{r}\leq t+1$, let
%		$$\omega_m=\mu~\textrm{and}~l_m=n_{j_1i_1}.$$
%		\item For $m=(S_{2k+1}+r)(t+1)+\widetilde{r}$, $k\geq0$, $0\leq r<N_{2k+2}$, $1\leq \widetilde{r}\leq t+1$,
%		\begin{itemize}
%			\item if $1\leq r\leq t-1$, let
%			$$\omega_m=\mu~\textrm{and}~l_m=n_{j_1i_1};$$
%			\item if $r=t$, let
%			$$\omega_m=\mu~\textrm{and}~l_m=n_{j_1i_2};$$
%			\item if $r=t+1$, let
%			$$\omega_m=\nu~\textrm{and}~l_m=n_{j_2i_1}.$$
%		\end{itemize}
%		
%		
%	\end{itemize}
	
For each $m\geq 1$ put $\widetilde{\Gamma}_m=\widetilde{\Gamma}_L^{\omega_m}$ and select an arbitrary point $x_m\in \widetilde{\Gamma}_m$. Clearly, each $\widetilde{\Gamma}_m$ is
a finite set and when $\omega_m=\mu$,
the choice of $x_m$ is not unique. In what follows, performing the above procedure, we construct a family of sequences rather than one concrete sequence.
Set
	\begin{equation}\label{selection-of-M-m}
	M_m=mL+\sum_{i=1}^ml_i
	\end{equation}	
and define a sequence $\{w_u\}_{u=0}^\infty$ as a concatenation of the following blocks.

First put $a=M_{(S_{2k}+r)(t+1)}$ and $b=M_{(S_{2k}+r+1)(t+1)}-1$ for some $k\geq 0$
and $0\leq r<N_{2k+1}$. Then the segment $\{w_u\}_{u=a}^b$ is defined by
\begin{equation}\label{segment-one}
		\begin{split}
		&x_{(S_{2k}+r)(t+1)+1},\cdots,f^{L-1}(x_{(S_{2k}+r)(t+1)+1}),y_{j_1i_1},\cdots,f^{n_{j_1i_1}-1}(y_{j_1i_1}),\\
		&x_{(S_{2k}+r)(t+1)+2},\cdots,f^{L-1}(x_{(S_{2k}+r)(t+1)+2}),y_{j_1i_1},\cdots,f^{n_{j_1i_1}-1}(y_{j_1i_1}),\\
		&\qquad\qquad\qquad\qquad\vdots\\
		&x_{(S_{2k}+r+1)(t+1)},\cdots,f^{L-1}(x_{(S_{2k}+r+1)(t+1)}),y_{j_1i_1},\cdots,f^{n_{j_1i_1}-1}(y_{j_1i_1}).\\
		\end{split}
		\end{equation}
Complementary segments of the sequence, for indexes between $a=M_{(S_{2k+1}+r)(t+1)}$ and $b=M_{(S_{2k+1}+r+1)(t+1)}+1$ for some $k\geq 0$
and $0\leq r<N_{2k+1}$ have a more complicated structure. We define them by:
\begin{equation}\label{segment-two}
\begin{split}
&x_{(S_{2k+1}+r)(t+1)+1},\cdots,f^{L-1}(x_{(S_{2k+1}+r)(t+1)+1}),y_{j_1i_1},\cdots,f^{n_{j_1i_1}-1}(y_{j_1i_1}),\\
&\qquad\qquad\qquad\qquad\vdots\\
&x_{(S_{2k+1}+r)(t+1)+t-1},\cdots,f^{L-1}(x_{(S_{2k+1}+r)(t+1)+t-1}),y_{j_1i_1},\cdots,f^{n_{j_1i_1}-1}(y_{j_1i_1}),\\
&x_{(S_{2k+1}+r)(t+1)+t},\cdots,f^{L-1}(x_{(S_{2k+1}+r)(t+1)+t}),y_{j_1i_2},\cdots,f^{n_{j_1i_2}-1}(y_{j_1i_2}),\\
&x_{(S_{2k+1}+r+1)(t+1)},\cdots,f^{L-1}(x_{(S_{2k+1}+r+1)(t+1)}),y_{j_2i_1},\cdots,f^{n_{j_2i_1}-1}(y_{j_2i_1}).\\
\end{split}
\end{equation}
%
%The segment $\{z_u\}_{u=M_{(S_{2k+1}+r)(t+1)}}^{M_{(S_{2k+1}+r+1)(t+1)}-1}$, $k\geq0$, $0\leq r<N_{2k+2}$, is defined as
%		\begin{equation}\label{segment-two}
%		\begin{split}
%		&x_{(S_{2k+1}+r)(t+1)+1},\cdots,f^{L-1}x_{(S_{2k+1}+r)(t+1)+1},y_{j_1i_1},\cdots,f^{n_{j_1i_1}-1}y_{j_1i_1},\cdots\cdots,\\
%		&x_{(S_{2k+1}+r)(t+1)+t-1},\cdots,f^{L-1}x_{(S_{2k+1}+r)(t+1)+t-1},y_{j_1i_1},\cdots,f^{n_{j_1i_1}-1}y_{j_1i_1},\\
%		&x_{(S_{2k+1}+r)(t+1)+t},\cdots,f^{L-1}x_{(S_{2k+1}+r)(t+1)+t},y_{j_1i_2},\cdots,f^{n_{j_1i_2}-1}y_{j_1i_2},\\
%		&x_{(S_{2k+1}+r+1)(t+1)},\cdots,f^{L-1}x_{(S_{2k+1}+r+1)(t+1)},y_{j_2i_1},\cdots,f^{n_{j_2i_1}-1}y_{j_2i_1}.
%		\end{split}
%		\end{equation}
		
By the selection of points $\{x_m\}_{m\geq1}$ and the definition of $y_{i_1i_1},y_{j_1i_2},y_{j_2i_1}$ it is easy to see that $\{w_u\}_{u=0}^{\infty}$ constitutes a $\gamma$-pseudo-orbit. Thus by the shadowing property, $\{w_u\}_{u=0}^{\infty}$  is $\zeta$-shadowed by some point in $X$.

For $d\geq1$ and any choice of $z_1\in \widetilde{\Gamma}_1,\ldots, z_d\in \widetilde{\Gamma}_d$, denote
\begin{eqnarray*}
&&G_d(z_1,\ldots,z_d):=\bigcap_{m=1}^df^{-M_{m-1}}\overline{B_L(z_m,\zeta)},\\
&&G_d:=\bigcap_{m=1}^d\bigcup_{x_m\in \widetilde{\Gamma}_m}f^{-M_{m-1}}\overline{B_L(x_m,\zeta)}.
\end{eqnarray*}
Observe that both $G_d(z_1,\ldots,z_d)$ and $G_d$ are nonempty closed sets. Moreover, for $d\geq 1$, we have inclusions
\begin{eqnarray*}
&&G_d(z_1,\ldots,z_d)\supset G_{d+1}(z_1,\ldots,z_d,z_{d+1})\\
&&G_d\supset G_{d+1}.
\end{eqnarray*}
which implies that the following sets are closed and nonempty:
\begin{eqnarray*}
&&G(z_1,z_2,\ldots):=\bigcap_{d\geq1}G_d(z_1,\ldots,z_d)\\
&&G:=\bigcap_{d\geq1}G_d.
\end{eqnarray*}
	
Let $x_d,y_d\in \widetilde{\Gamma}_d$ with $x_d\neq y_d$. Since $x_d$ and $y_d$ are $(L,\eps)$-separated, there exists $0\leq l<L$ such that $$d(f^l(x_d),f^l(y_d))>\eps,$$
therefore, if $x\in B_L(x_d,\zeta)$ and $y\in B_L(y_d,\zeta)$, then
\begin{eqnarray*}
d(f^l(x),f^l(y))&\geq& d(f^l(x_d),f^l(y_d))-d(f^l(x_d),f^l(x))-d(f^l(y_d),f^l(y))\\
&>&\eps-2\zeta>\frac{\eps}{3}
\end{eqnarray*}
by the selection of $\zeta$ in (\ref{selection-of-zeta}). This implies that
$$
G_d(z_1,\ldots,z_{d-1},x_d)\cap G_d(z_1,\ldots,z_{d-1},y_d)=\emptyset
$$
which shows that $G$ is a closed set which is a disjoint union of closed sets $G(x_1,x_2,\cdots)$ labeled by $(x_1,x_2,\cdots)$ with $x_m\in \widetilde{\Gamma}_m$ for each $m$. In other words, two different sequences label two different sets.

\textbf{Step 3}. Next we are going to prove that	
\begin{equation}\label{G-is-irregular}
	G\subset I(f).
\end{equation}
	
%	Recall the special equivalent metric we choose on $M$ in (\ref{equivalent-metric}).
Let $z$ be a point corresponding to some closed set $z\in G(x_1,x_2,\cdots)$ with $x_m\in \widetilde{\Gamma}_m$. Then $z$ $\zeta$-shadows some $\gamma$-pseudo-orbit $\{w_u\}_{u=0}^{\infty}$
constructed above in Step~2.
Since for every $x,y\in X$ we have $\rho(\delta_x,\delta_y)=d(x,y)$,
thus it is easily seen that
	\begin{equation}\label{orbit-and-pseudo-orbit}
	\rho(\mathcal{E}_n(z),\frac{\sum_{u=0}^{n-1}\delta_{z_u}}{n})\leq \zeta,~\forall n\in \mathbb{N}.
	\end{equation}
	
Since  $x_m\in \widetilde{\Gamma}_m$, we have
	\begin{equation}\label{x-m}
	\rho(\mathcal{E}_L(x_m),\omega_m)\leq \frac{\alpha}{6(t+1)}
	\end{equation}
by the selection of open sets $F_1$ and $F_2$. Therefore, Lemma~\ref{lem:prohorov}\eqref{lem:prohorov:3} gives
\begin{eqnarray*}
	\rho(\mathcal{E}_{L+l_m}(x_m),\omega_m)&\leq& \rho(\mathcal{E}_{L}(x_m),\omega_m)+\rho(\mathcal{E}_{L}(x_m),\mathcal{E}_{L+l_m}(x_m))\\
	&\leq& \frac{\alpha}{6(t+1)}+\frac{2T}{L}.
\end{eqnarray*}
%	Note that the right hand of (\ref{longer-x-m}) is no greater than
%	$$\frac{\alpha}{6(t+1)}+\frac{T}{L}\cdot2,$$
%	which does not depend on $x_m$ and $l_m$.
	Thus, by the selection of $L$ in \eqref{chosen-of-L}, and  Lemma~\ref{lem:prohorov} we have
	\begin{eqnarray*}
		% \nonumber to remove numbering (before each equation)
		& & \rho(\mathcal{E}_{S_{2k-1}(t+1)L+\sum_{m=1}^{S_{2k-1}(t+1)}l_m}(z),\mu) \\
		&&\qquad\leq \frac{N_{2k-1}}{S_{2k-1}}(\frac{\alpha}{6(t+1)}+\frac{2T}{L})+2\frac{S_{2k-2}}{S_{2k-1}}+\zeta \\
		&&\qquad\leq \frac{\alpha}{6(t+1)}+\frac{2T}{L}+\frac{2}{\lambda}+\frac{\alpha}{12(t+1)} \\
		&&\qquad\leq \frac{\alpha}{6(t+1)}+\frac{\alpha}{24(t+1)}+\frac{\alpha}{24(t+1)}+\frac{\alpha}{12(t+1)} \\
		&&\qquad\leq \frac{\alpha}{3(t+1)}.
	\end{eqnarray*}
Next, if we put
$$
\beta=\frac{tL+(t-1)n_{j_1i_1}+n_{j_1i_2}}{(t+1)L+(t-1)n_{j_1i_1}+n_{j_1i_2}+n_{i_2j_1}}
$$
then by Lemma~\ref{lem:prohorov}, similarly as before we obtain that
	\begin{eqnarray*} % \nonumber to remove numbering (before each equation)
		& & \rho(\mathcal{E}_{S_{2k}(t+1)L+\sum_{m=1}^{S_{2k}(t+1)}l_m}(z),\beta \mu+(1-\beta)\nu) \\
		     &&\qquad\leq \frac{N_{2k}}{S_{2k}}\left[\beta\left(\frac{\alpha}{6(t+1)}+\frac{2T}{L}\right)+ %\\
%		     & &\qquad\qquad
(1-\beta)\left(\frac{\alpha}{6(t+1)}+\frac{2T}{L}\right)\right]+\frac{2S_{2k-1}}{S_{2k}}+\zeta \\
		&&\qquad\leq \frac{\alpha}{6(t+1)}+\frac{2T}{L}+\frac{2}{\lambda}+\zeta\\
%		&&\qquad\leq \frac{\alpha}{3(t+1)}.\\
%&&\qquad
%\leq\frac{\alpha}{6(t+1)}+\frac{2T}{L}+\frac{2}{2^{2k-1}}+\frac{\alpha}{12(t+1)} \\
&&\qquad\leq\frac{\alpha}{6(t+1)}+\frac{\alpha}{24(t+1)}+\frac{\alpha}{24(t+1)}+\frac{\alpha}{12(t+1)} \\
&&\qquad\leq\frac{\alpha}{3(t+1)}.
	\end{eqnarray*}
Now it is enough to observe that
\begin{eqnarray*}
	d(\mu,\beta \mu+(1-\beta)\nu)&=& \alpha(1-\beta)\geq \frac{L}{(L+T)(t+1)}\alpha\geq \frac{5\alpha}{6(t+1)}\\
	&>&\frac{\alpha}{3(t+1)}+\frac{\alpha}{3(t+1)},
\end{eqnarray*}
which shows that $\{\mathcal{E}_n(z)\}_{n\geq1}$ diverges, and so $z\in I(f)$ by Lemma~\ref{lem:quasi}.
%
%\textbf{Step 4}. Observe that if $x\in G(x_1,\ldots, x_m)$ and $y\in G(y_1,\ldots, y_m)$ and $x_k\neq y_k$ for some $k$
%then $B_{M_m}(x,\eps/6)\cap B_{M_m}(y,\eps/6)=\emptyset$. Therefore for each $m$ the set $G$ contains $(M_m,\eps/6)$-separated set with the number of elements not smaller than the number of sequence $x_1,\ldots,x_m$
%allowed in the construction of $G$. One notes that the frequency that $\omega_m$ appears as $\nu$ does not exceed $1/(t+1)$, thus the following inequality holds
%$$
%s_{M_m}(G,\eps/6)\geq \exp(\frac{t}{t+1}mL(h_{\mu}(f)-2\eta)).
%$$
%By the selection of $L$, we have
%	\begin{eqnarray}\label{L}
%	L(h_{\mu}(f)-2\eta)\geq (L+T)(h_{\mu}(f)-3\eta)\label{L-T}
%	\end{eqnarray}
%and also it is easy to see that
%\begin{equation}
%M_m=mL+\sum_{i=1}^ml_i\leq m(L+T).
%\label{eq422'}
%\end{equation}
%These inequalities along with (\ref{estimation-of-t}) give that
%$$\frac{t}{t+1}mL(h_{\mu}(f)-2\eta)\geq m(L+T)(h_{\mu}(f)-4\eta)\geq M_m(h_{\mu}(f)-4\eta).$$
%Since $s_n(G,\eps)$ is non-increasing in $\eps$. Combining these facts, one gets
%$$
%h_d(f,G)\geq \lim_{m\to \infty}\frac{s_{M_m}(G,\frac{\eps}{6})}{M_m}\geq h_{\mu}(f)-4\eta>\htop(f)-5\eta.
%$$
%%This completes the Step~4, finishing the proof for the case $\htop(f)<\infty$.

\textbf{Step 4}.
Let $h=h_{\mu}(f)-5\eta>\htop(f)-6\eta$. We are going to show that
     \begin{equation}\label{lower-bound-of-entropy}
       \htop(f,G)\geq h.
     \end{equation}
     Recall that  $\htop(f,G)$ is defined in \eqref{definition-of-topological-entropy} and we will follow the
     notation introduced there. Since $C(G;s,\sigma,f)$ is a non-increasing function of $\sigma$, it is enough to prove that there exists $\widetilde{\sigma}>0$ such that
     \begin{equation}\label{estimation-of-C}
       C(G;h,\widetilde{\sigma},f)\geq1.
     \end{equation}
We will prove that (\ref{estimation-of-C}) holds for $\widetilde{\sigma}=\frac{\eps}{6}$.
By definition,
$$C(G;h,\frac{\eps}{6},f)=\lim_{n\rightarrow\infty}C(G;h,n,\frac{\eps}{6},f)$$
and
  $$C(G;h,n,\frac{\eps}{6},f)=\inf_{\mathcal{C}\in\mathcal{G}_n(G,\frac{\eps}{6})}\sum_{B_u(z,\frac{\eps}{6})\in\mathcal{C}}e^{-hu}.$$
        Since $G$ is compact, we can find at least one finite cover $\mathcal{C}\in\mathcal{G}_n(G,\frac{\eps}{6})$ of $G$.
        We may also assume that if $B_u(z,\frac{\eps}{6})\in \mathcal{C}$ then $B_u(z,\frac{\eps}{6})\cap G\neq\emptyset$,
        since otherwise we may remove this set from $\mathcal{C}$ and it still remains a cover of $G$.
%        This makes sense since the element of $\mathcal{C}$ with $B_u(z,\frac{\eps}{6})\cap G=\emptyset$ does not contribute when taking infimum in the definition of $C(E;h,n,\frac{\eps}{6},f).$
It is enough to show that for sufficiently large $n$ (which affects the choice of $\mathcal{C}$) we have
   \begin{equation}\label{estimation-of-C-n}
     \sum_{B_u(z,\frac{\eps}{6})\in\mathcal{C}}e^{-hu}\geq 1.
   \end{equation}
Let $q>0$ be an integer such that
  \begin{equation}\label{selection-of-q}
    \frac{q+1}{q}\leq\frac{h_{\mu}(f)-3\eta}{h_{\mu}(f)-2\eta}\cdot\frac{h_{\mu}(f)-4\eta}{h_{\mu}(f)-5\eta}.
  \end{equation}
We claim that  condition \eqref{estimation-of-C-n} holds for any $n\geq M_q$.

 Let us fix any $n\geq M_q$ and any $\mathcal {C}\in\mathcal {G}_{n}(G,\frac{\eps}{6})$. Define the cover $\mathcal {C}'$ in which each ball $B_{u}(z,\frac{\eps}{6})\in \mathcal{C}$ is replaced by $B_{M_{m}}(z,\frac{\eps}{6})$, where $M_{m}\leq u< M_{m+1}$, $m\geq q$ and $M_m$ are defined in \eqref{selection-of-M-m}. Then clearly
\begin{equation}\label{contract-prolong}
  \sum_{B_{u}(x,\frac{\eps}{6})\in\mathcal {C}}e^{-hu}\geq\sum_{B_{M_{m}}(x,\frac{\eps}{6})\in\mathcal {C}'}e^{-hM_{m+1}}.
\end{equation}
%Consider a specific $\mathcal {C}'$ and l
Let $c$ be the largest value of $m$ for which there exists $B_{M_{m}}(z,\frac{\eps}{6})\in\mathcal {C}'.$ Define
$$
\mathcal {V}_{c}:=\bigcup_{l=1}^{c}\mathcal {W}_{l} \quad \textrm{ where }\quad \mathcal {W}_{l}:=\prod_{i=1}^{l}\widetilde{\Gamma}_{i}
$$
and sets $\widetilde{\Gamma}_{i}$ are defined in Step~2 just before \eqref{selection-of-M-m}.
For every $1\leq j\leq k$ and $v\in\mathcal {W}_{j}$ we say that $v$ is a \emph{prefix} of $w\in\mathcal {W}_{k}$ if the first $j$ coordinates of $w$ coincide with $v$. Note that each $p\in\mathcal {W}_{m}$ is a prefix of exactly $|\mathcal {W}_{c}|/|\mathcal {W}_{m}|$ words of $\mathcal {W}_{c}$. If $\mathcal {W}\subset\mathcal {V}_{c}$ contains a prefix of each word of $\mathcal {W}_{c}$, then
$$\sum_{m=1}^{c}|\mathcal {W}\cap\mathcal {W}_{m}|\frac{|\mathcal {W}_{c}|}{|\mathcal {W}_{m}|}\geq |\mathcal {W}_{c}|.$$
To see this, it is enough to note that for each word of $\mathcal{W}_c$, one of its prefixes must be contained in  $\mathcal{W}\cap\mathcal {W}_{m}$ for some $1\leq m\leq c$. Moreover, the number of words in $\mathcal {W}_{c}$ with this prefix does not exceed $|\mathcal {W}_{c}|/|\mathcal {W}_{m}|$.
Thus if $\mathcal {W}$ contains a prefix of each word of $\mathcal {W}_{c}$, then
\begin{equation}\label{w}
  \sum_{m=1}^{c}|\mathcal {W}\cap\mathcal {W}_{m}|/|\mathcal {W}_{m}|\geq1.
\end{equation}
Note that since $\mathcal {C}'$ is a cover of $G$, $\mathcal{C}'$ corresponds to a subset $\mathcal {W}\subset\mathcal{V}_c$ such that $\mathcal{W}$ contains a prefix of each word of $\mathcal {W}_{c}$.
%Moreover, there is a one-to-one corresponding between $\mathcal {W}\cap\mathcal {W}_{m}$ and $B_{M_{m}}(z,\eps)\in\mathcal {C}'$.
More precisely, as discussed at the end of Step 2, each $x\in B_{M_{m}}(z,\frac{\eps}{6})\cap G$ corresponds to a closed set $G(x_1,x_2,\cdots,x_m)\supset G$, and thus corresponds to a point in $\mathcal {W}_{m}$. Moreover, sequence $x_1,\ldots, x_m$ is uniquely defined since, if the converse is true, it is not hard to see that there exist $1\leq i\leq m$ and $x_i,y_i\in \widetilde{\Gamma}_L^{\mu}$ with $x_i\neq y_i$ such that
$$
f^{M_{i-1}}(z)\in B_L(x_i,\frac{\eps}{6}+\frac{\eps}{3})\cap B_L(y_i,\frac{\eps}{6}+\frac{\eps}{3})
$$
and this leads to:
 $$
 d_L(x_i,y_i)\leq d_L(x_i,f^{M_{i-1}}(z))+d_L(y_i,f^{M_{i-1}}(z))< \eps,
 $$
 which contradicts the fact that $x_i$ and $y_i$ are $(L,\eps)$-separated.

 Thus (\ref{w}) gives
\begin{equation}\label{estimation-of-w-m}
  \sum_{B_{M_{m}}(x,\eps)\in\mathcal {C}'}\frac{1}{|\mathcal{W}_{m}|}\geq1.
\end{equation}

By the selection of the sequence $\widetilde{\Gamma}_i$ we see that the average number of occurrences of $\widetilde{\Gamma}_L^{\nu}$ does not exceed $1/(t+1)$ and therefore we have
\begin{equation}\label{ine-one}
 |\mathcal {W}_{m}|\geq \left(e^{L(h_{\mu}(f)-2\eta)}\right)^{\frac{m t}{t+1}}.
\end{equation}
Recall that by the selection of $L$ we have
\begin{eqnarray}\label{L}
% \nonumber to remove numbering (before each equation)
%  \frac{e^{L(h_{\mu}(f)-\eta)}}{s^2}&\geq& e^{L(h_{\mu}(f)-2\eta)}, \\
   L(h_{\mu}(f)-2\eta)&\geq& (L+T)(h_{\mu}(f)-3\eta),\label{L-T}\\
m(L+T)&\geq& M_m.\label{MmL-T}
\end{eqnarray}
Combining \eqref{ine-one}--\eqref{MmL-T} and \eqref{estimation-of-t} we have
$$
%\begin{equation}\label{ine-three}
  |\mathcal {W}_{m}|\geq e^{m(L+T)(h_{\mu}(f)-4\eta)}\geq e^{M_m(h_{\mu}(f)-4\eta)}.
$$
%  \end{equation}
which by \eqref{estimation-of-w-m} gives
\begin{equation}\label{h-star}
  \sum_{B_{M_{m}}(x,\eps)\in\mathcal {C}'}e^{-M_{m}(h_{\mu}(f)-4\eta)}\geq1.
\end{equation}

On the other hand, by the selection of $q$ in (\ref{selection-of-q}), definition \eqref{selection-of-M-m} and the inequality \eqref{L-T}, we easily check that for $m\geq q$,
%{\color{red}Can you explain in more detail how to get \eqref{estimation}}
$$\frac{M_{m+1}}{M_m}\leq \frac{(m+1)(L+T)}{mL}\leq \frac{q+1}{q}\cdot\frac{h_{\mu}(f)-2\eta}{h_{\mu}(f)-3\eta}\leq\frac{h_{\mu}(f)-4\eta}{h_{\mu}(f)-5\eta}$$
which implies that
\begin{equation}\label{estimation}
  hM_{m+1}=M_{m+1}(h_{\mu}(f)-5\eta)\leq M_m(h_{\mu}(f)-4\eta).
\end{equation}
%Hence, for $n\geq M_{q}$, we can use
Combining \eqref{contract-prolong}, \eqref{h-star} and \eqref{estimation}, we obtain that
$$\sum_{B_{u}(x,\frac{\eps}{6})\in\mathcal{C}}e^{-hu}\geq1,$$
and so the claim holds.

This implies that for any $n\geq M_q$ we have
$$C(G;h,n,\frac{\eps}{6},T)\geq1,$$
therefore \eqref{estimation-of-C} holds and thus \eqref{lower-bound-of-entropy} is proved.

We have just completed the proof of the case $\htop(f)<\infty$. For the case $\htop(f)=\infty$, it is enough to replace condition $\htop(f)-\eta<\htop(f|_\Lambda)$ in Step~1 by
$\htop(f|_\Lambda)>\eta$. All the other arguments in the proof remain unchanged. We leave the details to the reader.
\end{proof}

Now we are ready to prove our first main result.
It was proved in \cite{TKSM} that if a dynamical system with the shadowing property has a recurrent not minimal point, or a minimal sensitive point, then the entropy must be positive. A similar analysis will be performed here.

\begin{proof}[\textbf{Proof of Theorem~\ref{X-full}}]
The case of $\htop(f)>0$ is covered by Theorem~\ref{thm:shadowingpte}. For the remaining case, assume on the contrary that  $\htop(f)=0$
and $I(f)\neq \emptyset$, which by Lemma~\ref{lem:quasi} equivalently means that $Q(f)\neq X$. By Lemma~\ref{lem:quasi} there is $x\in X$ be such that $\mathcal{E}_n(x)$ diverges. Take two distinct invariant measures $\mu,\nu$ that are accumulation points of the sequence $\mathcal{E}_n(x)$.  First assume that $\supp (\mu)\neq \supp(\nu)$. Without loss of generality we may assume that $\supp(\mu)\setminus\supp(\nu)\neq \emptyset$. Fix any recurrent point $z\in \supp(\nu)$. We claim that $z$ is a sensitive point. To see that, fix any $y\in \supp(\mu)$ and $\eps>0$ such that $B_{3\eps}(y)\cap \supp(\nu)=\emptyset$. Fix any $\eta\in (0,\eps)$, put $U=B_{\eta}(y)$, $V=B_{\eta}(z)$
and observe that $\mu(U)>0$ and $\nu(V)>0$. Thus, there is $\delta>0$ and increasing sequences $n_i,m_i$ such that
$\mathcal{E}_{n_i}(x)(U)>\delta$ and $\mathcal{E}_{m_i}(x)(V)>\delta$. In particular, there is a point $q\in X$ and $0<n<m$ such that
$$
d(q,z)<\eta, \quad d(f^m(q),z)<\eta, \quad d(f^n(q),f^n(z))\geq \dist(f^n(q),\supp(\nu))>\eps.
$$
Since $z$ is a recurrent point we can repeat arguments similar to those in Lemma~\ref{horseshoe}
and show that $\htop(f)>0$ (see also \cite{LO}). This is a contradiction, hence the only remaining possibility is that $\supp(\mu)=\supp(\nu)$. If $\supp(\mu)$ is not a minimal set, then we can replace $\nu$
by a measure supported on a minimal set $Y\subsetneq \supp(\mu)$ and repeat previous argument showing that $\htop(f)>0$. Therefore the only possibility is that $\supp(\mu)$ is a minimal set. By \cite{LO}, assumption $\htop(f)=0$ implies that there is no sensitive point in $\supp(\mu)$
hence $(\supp(\mu),f|_{\supp(\mu)})$ is an equicontinuous minimal system. But all such systems are uniquely ergodic (e.g. see \cite{Brown}) which gives $\mu=\nu$. This is again a contradiction, which completes the proof.
\end{proof}

\section{Proof for Theorem \ref{main-thm}}\label{proof-for-main}

We need the following lemma which is a simplified version of the shredding lemma in \cite{Abdenur-Andersson}. For convenience of the reader, we include a proof here. This is also helpful in our subsequent proof of Theorem~\ref{main-thm}, which needs some details of the construction.
\begin{lem}\label{shredding}
For every $\eps>0$ there exists a dense set $\mathcal{C}^\eps\subset H(M)$ (resp. $\mathcal{C}^\eps\subset C(M)$)
such that for very $f\in \mathcal{C}^\eps$ there are pairwise disjoint open sets $U_1,\ldots, U_N$ such that:
\begin{enumerate}[(i)]
\item\label{shredding:c0} each $U_j$ is a trapping region: $f(\overline{U_j})\subset U_j$, $j=1,\ldots,N$;

\item\label{shredding:c1} the union of sets $U_j$ occupies, Lebesgue-wise, most of $M$:
$$\Leb(\bigcup_{j=1}^N U_j)>1-\eps;$$

\item\label{shredding:c2} each $U_j$ is contained in a basin of attraction of a periodic cycle of sets, i.e. there are open sets $W_j^1,\ldots, W_j^{k_j}$ such that
\begin{enumerate}[(a)]
	\item\label{shredding:c2:a} $\diam(W_j^i)<\eps$ for $i=1,\ldots,k_j$;
	\item\label{shredding:c2:b} $f(\overline{W_j^i})\subset W_j^{i+1}$ for $i=1,\ldots,k_j-1$ and $f(\overline{W_j^{k_j}})\subset W_j^{1}$;
	\item\label{shredding:c2:c} $\overline{U_j}\subset \bigcup_{n\geq 0}f^{-n}(\bigcup_{i=1}^{k_j}(W_j^i)).$
\end{enumerate}
\end{enumerate}
\end{lem}

\begin{proof} \textbf{The case of $H(M).$}
We are going  to prove that for any $f\in H(M)$ and $\zeta>0$, there exists a map $g\in \mathcal{C}^{\eps}$ such that $d_H(f,g)<2\zeta$.

Fix any $f\in H(M)$ and any $\zeta,\eps>0$.
Let $0<\sigma<\min\{\zeta,\eps\}$ be such that $d(x,y)<\sigma$ implies that $d(f(x),f(y))<\zeta$.

We begin with a decomposition of $M$ (see Definition~\ref{def:decomposition}), i.e. we decompose $M$ into finitely many small pieces $\{R_i\}_{i\in I}$, say $\diam R_i<\sigma$, each homeomorphic to the simplex
$$\Delta_k=\{(x_1,x_2,\cdots,x_k)\in \mathbb{R}^k:\sum_{i=1}^kx_i\leq1~\textrm{and}~x_i\geq0~\textrm{for all}~i\}.$$

Note that $\Int(f(R_i))\neq \emptyset$ for every $i$.
	We claim that for each $i$
	there is $j$ such that $\Int(f(R_i))\cap \Int(R_{j})\neq \emptyset$. Assume on the contrary that it is not the case, that is  $\Int (f(R_i))\subset \cup_{k\in I}\partial R_k$.
Denote 	$U_1=\Int f(R_i)$ and observe that either $U_1\subset \partial R_1$ or there is $\alpha>0$ such that $U_2=U_1\cap \{x : d(x, \partial R_1)>\alpha\}\neq \emptyset$, and clearly also $U_2\subset \cup_{k\geq2}\partial R_k$.
As before, either $U_2\subset \partial R_2$ or there is a nonempty open (in $M$) set $U_3\subset \cup_{k\geq3}\partial R_k$. Since $I$ is a finite set, there is $k\in I$ and a nonempty open set
$U_k\subset M$ such that $U_k\subset \partial R_k$, which is a contradiction. Indeed the claim holds, and so we can define a function $\tau \colon I \to I$ (not necessarily surjective)
such that $\Int(f(R_i))\cap \Int(R_{\tau(i)})\neq \emptyset$ for every $i\in I$.

Once $\tau$ is chosen, we select for each $i\in I$ a point $p_i\in \Int(R_i)$ such that $f(p_i)\in \Int(R_{\tau(i)})$. For $\delta>0$, we write $R_i^{\delta}$ for the set $\{x\in R_i:d(x,\partial R_i)>\delta\}$. Since $f$ is a homeomorphism, we can choose $\delta>0$ small enough, so that each $p_i\in R_i^{\delta}$, $f(p_i)\in R_{\tau(i)}^{\delta}$ and moreover, $\Leb(\bigcup_{i\in I}R_i^{\delta})>1-\eps$.

Take $\delta'>0$ small enough to satisfy $\overline{f(B_{\delta'}(p_i))}\subset R_{\tau(i)}^{\delta}$ for every $i\in I$. We are going to construct a homeomorphism $h:M\to M$ that sends $R_i^{\delta}$ into $B_{\delta'}(p_i)$.  This can easily done by using the linear structure in each $R_i$ induced by the charts $\eta_i:R_i\to \Delta_k$.

To see this, fix any $x\in R_i$ and denote $\hat{x}=\eta_i(x)$. For any $x\neq p_i$
let $\hat{q}^x_i$ denote the unique point in $\partial \Delta_k$ on the ray starting at $\hat{p_i}$ and passing through $\hat{x}$.
Clearly the following map $\alpha\colon R_i\to \R$ is continuous
$$
\alpha(x)=\begin{cases}0, & x=p_i\\
\frac{d(\hat{x},\hat{p}_i)}{d(\hat{p}_i, \hat{q}^x_i)},&x\neq p_i
\end{cases}.
$$
Now, fix any (sufficiently large) $\beta\geq 1$ and define a homeomorphism $h_i\colon M\to M$ by
$$h_i(x)=\begin{cases}
x,&\textrm{if}~~x\notin R_i\cup \set{p_i}\\
p_i,&x=p_i\\
\eta_i^{-1}(\hat{p_i}+\alpha(x)^\beta(\hat{q}_i^x-\hat{p_i})),
&\text{otherwise}
\end{cases}.$$
Note that $\alpha|_{\partial R_i}=1$ and do $h_i|_{\partial R_i}=id$.
Denote by $h$ composition of all the maps $h_i$ in any order. Since
\begin{equation}\label{h-i-R-i}
  h_i(R_i)= h_i^{-1}(R_i)= R_i,
\end{equation}
we easily see that all $h_i$ commute, which means that the order of composition is irrelevant. It is also not hard to see that $h$ is a homeomorphism which is close to identity,
that is
$$
	d_H(h,\text{id})\leq \max_{i\in I}\diam (R_i)<\sigma.
$$
	Denote $g=f\circ h$ and observe $g\in H(M)$ since $h\in H(M)$.
It directly follows from definition of $h_i$ that for any $\delta'>0$, taking $\beta>1$ sufficiently large we can ensure inclusion
$$\overline{h(R_i^{\delta})}\subset B_{\delta'}(p_i)$$
and consequently
\begin{equation}\label{subset}
  \overline{g(R_i^{\delta})}\subset R_{\tau(i)}^{\delta}.
\end{equation}

Note that since $I$ is a finite set, there is finite number of periodic orbits $O_1,\cdots,O_N$ under action of $\tau$. Moreover, the trajectory of every $l\in I$ ends up in one of the orbits $O_i$ after at most $|I|$ iterations of $\tau$. Thus if we denote $\widehat{O}_i=\bigcup_{k=1}^{|I|}\tau^{-k}(O_i)$ for $i=1,\cdots,N$, then $I=\bigsqcup_{i=1}^N\hat{O}_i$, where $\bigsqcup$ denotes the disjoint union. For each $i=1,\ldots , N$, fix a point $z_i\in O_i$ and define
\begin{equation}
U_j=\bigcup_{i\in \hat{O}_j} R_i^{\delta}~~\textrm{ and }~~W_i^s=R_{\tau^{s}(z_i)}^{\delta}\label{eq536'}
\end{equation}
for $s=1,2,\ldots,k_i$, where $k_i=|O_i|$.

Observe that
$$
g(\overline{U_j})=g(\overline{\bigcup_{i\in \hat{O}_j} R_i^{\delta}})= \bigcup_{i\in \hat{O}_j} g(\overline{R_i^{\delta}})=\bigcup_{i\in \hat{O}_j} \overline{g(R_i^{\delta})}\subset \bigcup_{i\in \hat{O}_j} R_{\tau(i)}^{\delta}\subset \bigcup_{i\in \hat{O}_j} R_i^{\delta}=U_j,
$$
which is \eqref{shredding:c0}. It is also clear that
$$
\Leb(\bigcup_{j=1}^N U_j)=\Leb(\bigcup_{j=1}^N\bigcup_{i\in \hat{O}_j}R_i^{\delta})=\Leb(\bigcup_{i\in \cup_{j=1}^N\hat{O}_j}R_i^{\delta})=\Leb(\bigcup_{i\in I}R_i^\delta)>1-\eps
$$
which is exactly \eqref{shredding:c1}.
Note that $\diam W_i^s\leq \max_{i\in I}\diam (R_i)<\sigma<\eps$ which gives \eqref{shredding:c2:a}. By \eqref{subset} we have
$$
g(\overline{W_i^s})=g(\overline{R^{\delta}_{\tau^s(z_i)}})=\overline{g(R^{\delta}_{\tau^s(z_i)})}\subset R_{\tau^{s+1}(z_i)}^{\delta}.
$$
But if $s<k_i$ then $R_{\tau^{s+1}(i)}^{\delta}=W^{s+1}_i$ and when $s=k_i$ then $\tau^{s+1}(z_i)=\tau(z_i)$
and so in this case $R_{\tau^{s+1}(z_i)}^{\delta}=W^{1}_i$. This proves \eqref{shredding:c2:b}.
Finally, for every $i\in \hat{O}_j$ there is an $1\leq n\leq |I|$ such that $\tau^n(i)\in O_j$. Then
$$
g^n(\overline{R_i^\delta})\subset R^{\delta}_{\tau^n(i)}\subset \cup_{i=1}^{k_j} W_j^{i},
$$
which implies that $\overline{R_i^\delta}\subset g^{-n}(\cup_{i=1}^{k_j} W_j^{i})$.
But then $\overline{U_j}\subset \bigcup_{n\geq0} g^{-n}(\cup_{i=1}^{k_j} W_j^{i})$ which is \eqref{shredding:c2:c}.

Finally observe that by \eqref{h-i-R-i} for every $x\in M$ we have $d(x,h(x))<\sigma$ and $d(x,h^{-1}(x))<\sigma$
which implies by definition of $\sigma$ that for every $x\in M$ we have
$$d(f^{-1}(x),h^{-1}f^{-1}(x))<\sigma<\zeta\quad \textrm{ and }\quad d(f(x),f(h(x)))<\zeta.$$
It means that $d_H(f,g)<2\zeta$. which completes the proof of the case of $H(M)$.

%One easily checks that $V_i$ and $W_i^s$ suffice for our lemma.

% Moreover, $g$ satisfies the three properties in the lemma. It remains to show that $d_H(f,g)<\zeta$.
%
%Moreover, for any $x\in M$, by (\ref{h-i-R-i}),
%$$d(f^{-1}(x),h^{-1}f^{-1}(x))<\sigma<\zeta~\textrm{and}~d(x,h(x))<\sigma$$
%which implies $d(f(x),fh(x))<\zeta$.

\textbf{The case of $C(M)$}.
Observe that in the construction for $H(M)$ we in fact only require that if $f(R_i)\cap R_j\neq \emptyset$, then there exists a
point $p_i\in \Int R_i$ such that $f(p_i)\in \Int R_j$. Assume for a moment that this condition holds. This allows us to define $\tau$ and the same proof can be repeated.
For any such $f\in C(M)$, we define $h\in H(M)$ as above and put $g=f\circ h$ obtaining all desired sets $U_i$, $W_i^j$.
It is also clear that $g\in C(M)$ and $d_C(f,g)<\zeta$ since metric $d_C$ is less restrictive than $d_H$.

Observe that by Lemma 3.11 in \cite{KMO2014} for any $\zeta>0$, we can find a map $\hat{f}$ such that
if
$$
\hat{f}(R_i)\cap R_j\neq\emptyset \quad \Longrightarrow \quad \hat{f}(\Int R_i)\cap \Int R_j\neq\emptyset
$$
and $d_C(f,\hat{f})<\zeta$. Then we can find $g$ for $\hat{f}$ as sketched above, and $d_C(f,g)<2\zeta$. The proof is completed.
\end{proof}

\begin{proof}[\textbf{Proof of Theorem~\ref{main-thm}}]
%We will apply Lemma~\ref{shredding} to prove that for generic $f\in H(M)$ (resp. $f\in C(M)$), $\Leb(I(f))=0$.
%We set our proof into six steps which correspond to the six conclusions.
We divide the proof into four intermediate steps, to make it more readable.

\textbf{Step 1:} \textit{Construction of $\Lambda\subset M$ with full $\Leb-$measure.}

Observe that each $\mathcal{C}^{\eps}$ provided by Lemma~\ref{shredding} is open and dense, hence the following set is residual
$$\mathcal{R}=\bigcap_{n\in \mathbb{N}}\mathcal{C}^{\frac{1}{n}}.$$
%It remains to show that elements of $\mathcal{R}$ have convergent Birkhoff averages $\Leb$-almost everywhere.

Fix $f\in \mathcal{R}$ and
for each $n\in \mathbb{N}$, let $V_n$ be the union of the open set $U_i$ corresponding to $C^{\frac{1}{n}}$. Denote
\begin{equation}
\Lambda=\bigcap_{n\geq1}\bigcup_{m\geq n}V_m\label{def:lambda}
\end{equation}
and observe that since $\Leb(V_n)>1-\frac{1}{n}$ for each $V_n$, we have $\Leb(\Lambda)=1$.
%We now show that
%$$\Lambda\subset Q(f)\cap\Delta.$$

\textbf{Step 2: $\Lambda\subset Q(f)\cap \Delta$.}
The proof that $\Lambda\subset Q(f)$ is exactly the same as in \cite[Theorem 3.6]{Abdenur-Andersson}, therefore we leave the details to the reader.
For the rest of the proof, fix any $x\in \Lambda$ and take $\mu\in \M_f(X)$ such that $\lim_{n\to\infty}\mathcal{E}_n(x)=\mu$. We claim that $\mu$ is a SRB-like measure.

By definition, it is enough to show that for every $\epsilon>0$, the set
$$A_{\eps}(\mu)=\{y\in X:~d(\pw (y),\mu)<\eps\}$$
has positive Lebesgue measure. %Note that $\Lambda=\bigcap_{n\geq1}\bigcup_{m\geq n}V_m$.
Take an integer $n>0$ such that $\frac{1}{n}<\epsilon$. By definition of $\Lambda$ there exists an $m\geq n$ such that $x\in V_m$.

By the definition of $V_m$, which were obtained in  Lemma~\ref{shredding} by \eqref{eq536'}, there exist $|I(m)|$ open sets $R_{m,i}^{\delta}$ with $\diam(R_{m,i}^{\delta})<\frac{1}{m}<\epsilon$
and a function $\tau \colon I(m)\to I(m)$
such that
$$V_m=\bigsqcup_{1\leq i\leq |I(m)|}R_{m,i}^{\delta}\quad \textrm{ and }\quad f(\overline{R_{m,i}^{\delta}})\subset R_{m,\tau(i)}^{\delta}\quad \text{ for }1\leq i\leq |I(m)|.$$
Suppose that $x\in R_{m,j}^{\delta}$ and let $d_m=\max_{1\leq i\leq |I(m)|}\diam(R_{m,i}^{\delta})$.
%Without loss of generality we may assume that $f^{t s}(x)\in R_j^\delta$
%for some integer $s>0$ and every integer $t>0$. Then there is $\alpha>0$ such that $\E_n(x)(R_j^\delta)>\alpha$ for every $n$ sufficiently large
%which implies that $\mu(\overline{R_j^\delta})\geq \limsup_{n\to\infty} \E_n(x)(\overline{R_j^\delta})>\alpha$.

Observe that if we fix any $z,w\in {R_{m,j}^{\delta}}$ and any $l\geq 1$ then
$$
f^l(z)\in f^l(R_{m,j}^{\delta}),f^l(w)\in f^l(R_{m,j}^{\delta})\quad \textrm{ and }\quad f^l(R_{m,j}^{\delta})\subset R_{m,\tau^l(j)}^{\delta}.
$$
Therefore,
%\begin{fact}\label{easy fact}
for any $z,w\in R_{m,j}^{\delta}$ and any $l\geq 1$ we have
\begin{equation}
\label{easy fact}
d(f^l(z),f^l(w))\leq d_m<\frac1m<\epsilon.
\end{equation}
%\end{fact}
This by Lemma~\ref{lem:prohorov} implies that for any $y\in R_{m,j}^{\delta}$ and any $n\geq 1$ we have (recall that $x\in R_{m,j}^\delta$):
$$\rho(\mathcal{E}_n(x),\mathcal{E}_n(y))\leq d_m<\epsilon$$
By the definition of $\mu$ and $\pw (y)$, there exists an integer $k>0$ such that
$$d(\mu,\E_k(x))<\frac{1}{2}(\eps-d_m)\quad \textrm{and}\quad d(\pw (y),\E_k(y))<\frac{1}{2}(\eps-d_m).$$
which yields that
\begin{eqnarray*}
% \nonumber to remove numbering (before each equation)
  d(\pw (y),\mu) &\leq& d(\pw (y),\E_k(y))+d(\E_k(y),\E_k(x))+d(\E_k(x),\mu) \\
   &<& \frac{1}{2}(\eps-d_m)+d_m+\frac{1}{2}(\eps-d_m)=\eps.
\end{eqnarray*}
which implies that $y\in A_\eps(\mu)$, and hence, since $y\in R_{m,j}^\delta$ was arbitrary, we have $R_{m,j}^{\delta}\subset A_{\eps}(\mu).$
Recall that $R_{m,j}^\delta$ is open, therefore
$$\Leb(A_{\eps}(\mu))\geq \Leb(R_{m,j}^{\delta})>0$$
which completes our proof.
%-----
%Observe that by Lemma~\ref{shredding} for every $\eps>0$ there is $k$ such that every $x\in \Lambda$ after at most $k$ iterations
%enters a set $W$ with $\diam (f^i(W))<\eps$ for every $i$ and $f^k(W)\subset W$ and there are finitely many such sets $W$.
%But then there is $K_\eps$ such that any $(n,\eps)$-spanning set contained in $\Lambda$ has at most $K_\eps$ elements.
%This implies that $\htop(f,\Lambda)=0$ proving Theorem~\ref{main-thm}\eqref{main-thm:2}.
%-----

\textbf{Step 3:} $\htop(f,\Lambda)=0$.

Recall the process of construction from Lemma~\ref{shredding}. For any $n\in \mathbb{N}$ there are $|I(n)|$ small open sets $R_{n,i}^{\delta}$ with $\diam(R_{n,i}^{\delta})<\frac{1}{n}$ such that (\ref{subset}) holds. So if for each $i$ we fix a point $x_{n,i}\in R_{n,i}^{\delta}$, then  by \eqref{easy fact} for any $l\geq 1$ the set $\{x_{n,i}\}_{i=1}^{I(n)}$ constitutes an $(l,\frac{1}{n})$-spanning set of $V_n$.
%Fix any $\sigma>0$ and let $m>n$ be such that $1/m<\sigma$.
%Then the set $V_m$ can be covered by Bowen balls $\{B_{I(n+m)}(x_i,\sigma)\}_{i=1}^{I(m)}$.

On the other hand, we can %decompose $M$ into smaller sets with diameter less than $\frac{1}{n+1}$ dividing sets $R_i$
replace $I(n)$ by its subsequence, and since diameters of $R_i^\delta$ can be arbitrarily small,
without loss of generality, we may assume that $|I(n)|< |I(n+1)|$ for each $n\geq 1$ and therefore,
\begin{equation}\label{I-n}
  |I(n)|\geq n \quad \text{ for } n\geq 1.
\end{equation}

Fix any $\sigma>0$ and select an integer $k\in \mathbb{N}$ such that $\frac{1}{k}<\sigma$. Since $\Lambda=\bigcap_{n\geq1}\bigcup_{m\geq n}V_m$, we have in particular that $\Lambda\subset \bigcup_{m\geq k}V_m$. For any $m,n\in\N$, $V_m$ can be covered by $|I(m)|$ Bowen balls $\{B_{|I(n+m)|}(x_{m,i},\frac1m)\}_{i=1}^{|I(m)|}$, so we have
\begin{eqnarray*}
% \nonumber to remove numbering (before each equation)
  C(\Lambda;t,n,\sigma,f)&=&\inf_{\mathcal {C}\in \mathcal {G}_{n}(\Lambda,\sigma)}\sum_{B_{u}(x,\sigma)\in \mathcal {C}}e^{-tu} \\
    &\leq& \sum_{m\geq k}|I(m)|e^{-t|I(n+m)|} \\
    &\leq& \sum_{m\geq k}|I(n+m)|e^{-t|I(n+m)|} \\
    &\leq& \sum_{s\geq |I(n+k)|}se^{-ts} \\
    &=& \frac{|I(n+k)|e^{-t|I(n+k)|}}{1-e^{-t}}+\frac{e^{-t(|I(n+k)|+1)}}{(1-e^{-t})^2}.
\end{eqnarray*}

By our assumptions, when $n\to \infty$ then also $|I(n)|\to \infty$. Thus we see that for any $t>0$,
$$C(\Lambda;t,\sigma,f):=\lim_{n\rightarrow\infty}C(\Lambda;t,n,\sigma,f)=0,$$
which implies that
$$\htop(\Lambda;\sigma,f):=\inf\{t:C(\Lambda;t,\sigma,f)=0\}=0.$$
This proves that
$$\htop(f,\Lambda)=\lim_{\sigma\rightarrow0} \htop(\Lambda;\sigma,f)=0.$$

\textbf{Step 4:}
%To complete the proof of Theorem~\ref{main-thm}\eqref{main-thm:1},
On one hand it is known (see \cite{KMO2014} and \cite{ Koscielniak2,Koscielniak}, respectively) that shadowing is a generic property in $C(M)$  and in $H(M)$. On the other hand, \cite{Koscielniak,Yano1980} show that for generic maps in $C(M)$ (resp. $H(M))$ we have $\htop(f)=\infty$. Thus by Theorem \ref{X-full}, we obtain that
$$
\htop(f,I(f))=\infty
$$
completing the proof.
\end{proof}

\section*{Acknowledgment}
The authors would like to thank anonymous referee for careful reading of the paper and valuable suggestions and comments.

The authors are grateful to Jian Li and Dominik Kwietniak for their numerous remarks and fruitful discussions on the relations
between entropies considered in the paper and topological properties of $\M(X)$.

Research of P. Oprocha was partly supported by the project ``LQ1602 IT4Innovations excellence in science'' and AGH local grant 10.420.003.

The research of X. Tian and Y. Dong was  supported by National Natural Science Foundation of China (grant no. 11671093).

\end{document}